\documentclass[11pt]{amsart}
\usepackage{a4wide,amsthm,amsmath,amssymb,color}
\usepackage{comment}
\usepackage{tikz,pgfplots}
\usetikzlibrary{arrows,snakes,backgrounds}
\usepackage[export]{adjustbox}
\usepackage[nocompress]{cite}
\theoremstyle{plain}             
\newtheorem{theorem}{Theorem}[section]
\newtheorem{lemma}[theorem]{Lemma}

\newtheorem{corollary}[theorem]{Corollary}

\newtheorem{remark}[theorem]{Remark}

\newcommand{\isdef}{\mathrel{\mathrel{\mathop:}=}}

\newcommand{\E}{\mathbb{E}}
\newcommand{\V}{\mathbb{V}}

\newcommand{\Cov}{\operatorname{Cov}}

\renewcommand{\div}{\operatorname{div}}

\newcommand{\balpha}{{\boldsymbol{\alpha}}}
\newcommand{\bbeta}{{\boldsymbol{\beta}}}
\newcommand{\bgamma}{{\boldsymbol{\gamma}}}

\newcommand{\bphi}{{\boldsymbol{\varphi}}}

\newcommand{\refd}{{\operatorname*{ref}}}
\newcommand{\essinf}{\operatorname*{ess\,inf}}
\newcommand{\esssup}{\operatorname*{ess\,sup}}
\renewcommand{\d}{\operatorname{d}\!}

\newcommand{\vertiii}[1]{{\vert\kern-0.25ex\vert\kern-0.25ex\vert #1 
    \vert\kern-0.25ex\vert\kern-0.25ex\vert}}
    \newcommand{\bvertiii}[1]{{\big\vert\kern-0.25ex\big\vert\kern-0.25ex\big\vert #1 
    \big\vert\kern-0.25ex\big\vert\kern-0.25ex\big\vert}}
    \newcommand{\Bvertiii}[1]{{\Big\vert\kern-0.25ex\Big\vert\kern-0.25ex\Big\vert #1 
    \Big\vert\kern-0.25ex\Big\vert\kern-0.25ex\Big\vert}}
        \newcommand{\bbvertiii}[1]{{\bigg\vert\kern-0.25ex\bigg\vert\kern-0.25ex\bigg\vert #1 
    \bigg\vert\kern-0.25ex\bigg\vert\kern-0.25ex\bigg\vert}}

\title[A note on the domain mapping method with rough diffusion coefficients]
{A note on the domain mapping method with rough diffusion coefficients} 
\author{M.\ D.\ Multerer}
\address{
Michael D.\ Multerer. Institute of Computational Science, Universit{\`a} della Svizzera italiana, Switzerland.
}
\email{michael.multerer@usi.ch}
\begin{document}

\maketitle
\begin{abstract}
In this article, we consider elliptic diffusion problems on random domains with
non-smooth diffusion coefficients. We start by illustrating the problems that
arise from a non-smooth diffusion coefficient by recapitulating the corresponding
regularity analysis. Then, we propose an alternative approach to address this
problem by means of a perturbation method. Based on the assumption that the 
diffusion coefficient can be decomposed in a possibly deterministic, analytic part and 
a rough random perturbation, we derive approximation results in terms of the 
perturbations amplitude for the approximation of quantities of interest of the 
solution. Numerical examples are given in order to validate and quantify the
theoretical results.
\end{abstract}

\section{Introduction}
Often, problems arising in science and engineering can be modeled in terms of 
boundary value problems. In general, the numerical solution of the latter is well 
understood if all input parameters are known exactly. 
In practice, however, this might be a too strong assumption if 
input parameters are only known up to certain measurement tolerances. 
In this view, particularly the treatment of uncertainties in the computational 
domain has become of growing interest, see e.g.~\cite{CK,HSS,%
xiu2,xiu,CNT16,HPS16,HSSS18}. In this article, we consider the elliptic diffusion equation
\begin{equation}\label{SPDE}
  -\div\big(a(\omega)\nabla u(\omega)\big) = f\ \text{in $D(\omega)$}, \quad
  	u(\omega) = 0\ \text{on $\partial D(\omega)$},
\end{equation}
as a model problem where the underlying domain 
$D\subset\mathbb{R}^d$ and the diffusion coefficient \(a(\omega)\)
are assumed to be random. 
This model may be used to account for
tolerances in the shape of products fabricated by line 
production or shapes which stem from inverse problems, 
like e.g.~tomography. 
The latter has recently been discussed, 
in the case of a deterministic diffusion coefficient, in \cite{GP16}.

Besides the fictitious domain approach 
considered in \cite{CK}, one might essentially distinguish two 
approaches to deal with uncertain domains: the \emph{perturbation method}, 
see e.g.\ \cite{HSS} and the references therein, 
which is suitable to handle small perturbations
of the nominal shape and the \emph{domain mapping method}, see e.g.\ \cite{xiu}.
The perturbation approach has also successfully been applied to deal with random data, like
random diffusion coefficients and random loadings,
see e.g.\ \cite{CS13,HPS13a,BC02,K64,KH92,BNK16}.

In this work, we shall combine the domain mapping method for the numerical treatment of
the random domain with the perturbation approach for dealing with the random coefficient.
A similar hybrid approach has recently 
been considered in \cite{CNT17} to account for uncertainties induced by the computational domain.
There, the domain mapping method is employed to capture the large deformations of the domain,
while the perturbation approach is used to account for the smaller deformations.

Usually, if the solution \(u({\bf x},\omega)\) 
to \eqref{SPDE} provides sufficient regularity, quantities of interest, like
expectation and variance, are computed by sophisticated sparse quadrature 
and quasi-Monte Carlo methods, see e.g.\ \cite{HHPS18,NTW08a,NTW08b,DKL+14,Caf98,SZ17}.
These methods alleviate the computational burden that comes along with the high dimensionality
inherent to this class of problems.
The analysis recently published in \cite{CNT16,HPS16,HSSS18} shows that such regularity results
are also available for the solution to \eqref{SPDE}, if the underlying data, i.e.\
the diffusion coefficient, the loading and possible data boundary data are analytic functions.
If this is not the case, one has to resort to the only slowly converging Monte Carlo method.
In practice, depending on the resolution of the discretization and the desired accuracy for the
quantities of interest, this is a strong limitation. Indeed, due to the lack of smoothness in 
the data the reference solution for the numerical example for this article had to be computed 
on a supercomputer expending an immense amount of resources. 
Therefore, our goal in this work is to weaken this requirement.  

Exemplarily, we focus
here on the diffusion coefficient and emphasize that other data can be handled in a similar way.
Assuming an \emph{essentially smooth} diffusion coefficient, i.e. a diffusion coefficient of the form
\[a=a_s+\varepsilon a_r,\quad 0 < \varepsilon\ll 1,\] where 
\(a_s\) is an analytic function and \(a_r\) is essentially
bounded, i.e.\ \(\|a_r\|_{L^\infty}\leq 1,\) 
we will derive approximation results for quantities of interest
of the solution \(u\) in terms of the perturbation's amplitude.

The rest of this article is organized as follows. In Section
\ref{sec:Problemform}, we introduce some basic definitions and
introduce the domain mapping method. Section~\ref{sec:analyticDiffusion}
adapts end extends the regularity results from \cite{HPS16} for the situation of an
analytic and random diffusion coefficient. The subsequent Section~\ref{sec:roughDiffusion}
is the main contribution of this article.
Here, we consider the case of rough diffusion coefficients in the domain mapping
framework and derive approximation results for the perturbation approach under consideration.
Afterwards, Section~\ref{sec:numericalRealization} 
gives a brief overview of the numerical realization of the presented
method. In particular, we explain how diffusion problems on random domains can efficiently
be realized by means of finite element methods and how a diffusion coefficient given in spatial
coordinates can be handled.
Finally, Section~\ref{sec:numres} provides a numerical example to validate the 
theoretical results.

\section{Problem formulation}\label{sec:Problemform}
In what follows, let $D_\refd\subset\mathbb{R}^d$ for \(d\in\mathbb{N}\) (of
special interest are the cases \(d=2,3\)) be a
domain with Lipschitz continuous boundary \(\partial D_\refd\) and let 
$(\Omega,\mathcal{F},\mathbb{P})$ 
be a probability space with $\sigma$-field $\mathcal{F}\subset 2^\Omega$ 
and a complete probability measure $\mathbb{P}$, i.e.~for all $A\subset B$ 
and $B\in\mathcal{F}$ with $\mathbb{P}[B]=0$ it follows $A\in\mathcal{F}$. We
are interested in computing quantities of interest of the solution to the elliptic diffusion problem
\begin{equation}\label{eq:modProb}
\begin{aligned}
-\div\big(a({\bf x,\omega})\nabla u({\bf x},\omega)\big) &= f({\bf x})&&\text{in }D(\omega),\\
u({\bf x},\omega) &= 0&&\text{on }\Gamma(\omega)\isdef\partial D(\omega)
\end{aligned}
\end{equation}
for \(\mathbb{P}\)-almost every \(\omega\in\Omega\). Note that the case of non-homogeneous Dirichlet data
in the domain mapping method can always be reduced to the above homogeneous case, see e.g.\ \cite{GP16}.
Neumann problems can be treated as well, if the matching condition is satisfied for for \(\mathbb{P}\)-almost every \(\omega\in\Omega\).

In order to guarantee the well posedness of \eqref{eq:modProb}, 
we assume that all data, i.e.\ the diffusion coefficient \(a\) and 
the loading \(f\)
are defined with respect to the hold-all domain
\[
\mathcal{D}\isdef\bigcup_{\omega\in\Omega}D(\omega).
\]
The diffusion coefficient  \(a({\bf x},\omega)\) shall be uniformly elliptic, i.e.\ there exist
\(\underline{a},\overline{a}\in(0,\infty)\) such that
\begin{equation}\label{eq:uniElliptic}
\underline{a}\leq\essinf_{{\bf x}\in\mathcal{D}}a({\bf x},\omega)\leq\esssup_{{\bf x}\in\mathcal{D}}a({\bf x},\omega)\leq\overline{a}
\end{equation}
\(\mathbb{P}\)-almost surely.

We make the crucial assumption that the random variation in the coefficient \(a({\bf x,\omega})\) is independent
of the random variation in the domain \(D(\omega)\). We note that under the same constraint of independence, 
it is also possible to consider random loadings or even random boundary data.
In order to model the random domain, as in \cite{GP16,HPS16}, we assume the existence of a uniform \(C^1\)-diffeomorphism 
\({\bf V}\colon\overline{D_\refd}\times\Omega\to\mathbb{R}^d\), i.e.\ 
\begin{equation}\label{eq:unif}
\|{\bf V}(\omega)\|_{C^1(\overline{D_\refd};\mathbb{R}^d)},\|{\bf V}^{-1}(\omega)\|_{C^1(\overline{D_\refd};\mathbb{R}^d)}\leq C_{\operatorname{uni}}\end{equation} for \(\mathbb{P}\)-almost every \(\omega\in\Omega\), such that 
\[
D(\omega)={\bf V}(D_\refd,\omega).
\]
Applying the domain mapping approach, the variational formulation reads then:
\begin{equation}\label{eq:varForm}
\begin{aligned}
&\text{Given \(\omega\in\Omega\), find \(\hat{u}(\omega)\in H^1_0(D_\refd)\) such that}\\
&\qquad\int_{D_\refd}{\bf A}(\omega)\nabla\hat{u}(\omega)\nabla v\d{\bf X}=\int_{D_\refd}f_\refd(\omega) v\d{\bf X}\quad\text{for all }
v\in H^1_0({D_\refd}),
\end{aligned}
\end{equation}
where
\[
{\bf A}({\bf X},\omega)\isdef(a\circ{\bf V})({\bf X},\omega)\cdot({\bf V}'^\intercal{\bf V}')^{-1}({\bf X},\omega)\cdot\det{\bf V}'({\bf X},\omega)
\]
and
\[
f_\refd({\bf X},\omega)\isdef (f\circ{\bf V})({\bf X},\omega)\cdot\det{\bf V}'({\bf X},\omega).
\]
Herein, \({\bf V}'\) denotes the Jacobian of \(\bf V\). 
In addition, we make use of the convention that \({\bf X}\in D_\refd\) always refers to a material point, while
\({\bf x}\in\mathbb{R}^d\) denotes a spatial point.

Then, there is the following one-to-one correspondence between the solution \(u\) to \eqref{eq:modProb}
and the solution \(\hat{u}\) to \eqref{eq:varForm}. It holds
\begin{equation}\label{eq:solutionIdentity}
u({\bf x},\omega)=(\hat{u}\circ{\bf V}^{-1})({\bf x},\omega)\quad\text{and}\quad\hat{u}({\bf X},\omega) = (u\circ {\bf V})({\bf X},\omega),
\end{equation}
see e.g.\ \cite{HPS16}.

\section{Analytic diffusion coefficients}\label{sec:analyticDiffusion}
In this section, we shall briefly recall the essential regularity results for the solution 
\(\hat{u}\) given that \(a\) is an analytic function
and refer to \cite{CNT16,HPS16} for a more comprehensive discussion of this topic.

In order to approximate the solution \(\hat{u}\) numerically, one usually starts from a truncated
Karhunen-Lo\`eve expansion of the underlying random fields, i.e.
\begin{equation}\label{eq:KLexpansions1}
\begin{aligned}
a(\bf x,\omega) &= \E[a]({\bf x}) +\sum_{k=1}^{N}\sqrt{\lambda_k}\psi_k({\bf x})X_k(\omega),\\
{\bf V}({\bf X},\omega) & = \E[{\bf V}]({\bf X})+\sum_{k=1}^{M}\sqrt{\mu_k}\bphi_k({\bf X})Y_k(\omega)
\end{aligned}
\end{equation}
with families \(\{X_k\}_k\) and \(\{Y_k\}_k\) of uncorrelated and centered random variables. 
Moreover, \(\{\lambda_k,\psi_k\}_k\) denote the eigen pairs of the covariance operator associated to \(a\),
while \(\{\mu_k,\bphi_k\}_k\) denote the eigen pairs associated to the covariance operator of \({\bf V}\).

The Karhunen-Lo\`eve expansion exists if the underlying random field is square integrable, i.e.\ in \(L^2(\Omega;\mathcal{X})\)
for an appropriate Banach space \(\mathcal{X}\). We shall assume this in what follows for \(a\). The square integrability
of \({\bf V}\) is a straightforward consequence from \eqref{eq:unif}.

One possibility to
compute such a truncated Karhunen-Lo\`eve expansion numerically is to employ a pivoted Cholesky decomposition,
see e.g.\ \cite{HPS14} and the references therein.

As equation~\eqref{eq:KLexpansions1} already indicates, we will assume here that the diffusion coefficient is given in
spatial coordinates. In view of the subsequent analysis, this is the more challenging situation. 
Nevertheless, we emphasize that the presented analysis is also capable of dealing with a diffusion coefficient that
is represented in material coordinates, i.e.\ \(a({\bf X},\omega)\).

Here and in the following, we make the common assumption that the families \(\{X_k\}_k\)
and \(\{Y_k\}_k\)
are even independent and identically distributed. As a consequence, 
the two families are particularly independent with respect to each other. 
After a possible scaling, we have that the range of the random variables is 
\(\Gamma\isdef[-1,1]\). We further assume that the random variables exhibit densities with respect to the Lebesgue
measure, such that the corresponding push-forward measures are given by 
\[
\mathbb{P}\circ X^{-1} = \rho_X({\bf y})\d{\bf y}\quad\text{and}\quad\mathbb{P}\circ Y^{-1} = \rho_Y({\bf z})\d{\bf z},
\]
respectively. Herein, we have \({\bf y}=[y_1,\ldots,y_N]\in\Gamma^N\), while \({\bf z}=[z_1,\ldots,z_M]\in\Gamma^M\).
The densities are of product structure due to the independence of the random variables, i.e.\
\(\rho_X({\bf y})=\rho_{X,1}(y_1)\cdots\rho_{X,N}(y_N)\) and
\(\rho_Y({\bf z})=\rho_{Y,1}(z_1)\cdots\rho_{Y,M}(z_M)\), respectively.
Moreover, the centeredness yields
\begin{equation}\label{eq:centeredness}
\int_\Gamma y_k\rho_{X,k}(y_k)\d{y_k} = 0\quad\text{for }k=1,\ldots,N\quad\text{and}\quad
\int_\Gamma z_k\rho_{Y,k}(z_k)\d{z_k} = 0\quad\text{for }k=1,\ldots,M.
\end{equation}
Therefore, we can reparametrize the expansions from \eqref{eq:KLexpansions1} and write
\begin{equation}\label{eq:KLexpansions2}
\begin{aligned}
a(\bf x,{\bf y}) &= \E[a]({\bf x}) +\sum_{k=1}^{N}\sqrt{\lambda_k}\psi_k({\bf x})y_k,\quad{\bf y}\in\Gamma^N,\\
{\bf V}({\bf X},{\bf z}) & = \E[{\bf V}]({\bf X})+\sum_{k=1}^{M}\sqrt{\mu_k}\bphi({\bf X})z_k,\quad{\bf z}\in\Gamma^M.
\end{aligned}
\end{equation}
This yields the parametrized variational formulation
\begin{equation}\label{eq:varFormPara}
\begin{aligned}
&\text{Given \({\bf y}\in\Gamma^M\),  \({\bf z}\in\Gamma^N\), find \(\hat{u}({\bf y},{\bf z})\in H^1_0(D_\refd)\) such that}\\
&\qquad\int_{D_\refd}{\bf A}({\bf y},{\bf z})\nabla\hat{u}({\bf y},{\bf z})\nabla v\d{\bf X}=\int_{D_\refd}f_\refd({\bf z}) v\d{\bf X}\quad\text{for all }
v\in H^1_0({D_\refd}),
\end{aligned}
\end{equation}

The expectation of \(\hat{u}\) is given by the Bochner-type integral
\[
\E[\hat{u}]({\bf X})\isdef\int_{\Gamma^{N}}\int_{\Gamma^{M}}\hat{u}({\bf X},{\bf y},{\bf z})\rho_{X}({\bf y})\rho_{Y}({\bf z})
\d{\bf z}\d{\bf y}
\]
and its variance by
\[
\V[\hat{u}]({\bf X})\isdef\int_{\Gamma^{N}}\int_{\Gamma^{M}}\big(\hat{u}({\bf X},{\bf y},{\bf z})-\E[u]({\bf X})\big)^2
\rho_{X}({\bf y})\rho_{Y}({\bf z})\d{\bf z}\d{\bf y}.
\]
In a similar fashion, we can also compute other quantities of interest, e.g.\ \(\E[F(\hat{u})]\), where 
\[
F\colon H^1_0(D_\refd)\to\mathbb{R}\]
is 
a continuous and linear functional. 
Note that there exists a version of Fubini's theorem for Bochner integrals. It guarantees that the order of integration
in the expressions for the expectation and the variance can be interchanged, see \cite{HP57}.

As we have seen so far, the computation of quantities of interest results in very high dimensional quadrature problems. 
In order to solve these quadrature problems
efficiently, one usually exploits the smoothness of the solution \(\hat{u}\) with respect to the parameters \({\bf y}\in\Gamma^{N}\)
and \({\bf z}\in\Gamma^{M}\):

Let all eigenfunctions \(\psi_k\) of the diffusion coefficient \(a\) be analytic, i.e.\ 
\begin{equation}\label{eq:analyticEF}
\|\partial^\balpha_{\bf x}\psi_k\|_{L^\infty(\mathcal{D})}\leq\balpha!\rho^{-|\balpha|}c_\psi,\quad \balpha\in\mathbb{N}^d,
\end{equation}

for \(\rho\in (0,1]\) and \(c_\psi>0\) uniformly in \(k\). Moreover, we introduce the quantities
\begin{equation}\label{eq:gammaDef}
\begin{aligned}
\bgamma_a&\isdef\Big[\big\|\sqrt{\lambda_1}\psi_1\big\|_{L^\infty(\mathcal{D})},\ldots,\big\|\sqrt{\lambda_N}\psi_N\big\|_{L^\infty(\mathcal{D})}\Big],\\
\bgamma_{\bf V}&\isdef\Big[\big\|\sqrt{\mu_1}\bphi_1\big\|_{W^{1,\infty}(D_\refd;\mathbb{R}^d)},\ldots,\big\|\sqrt{\mu_M}\bphi_M\big\|_{W^{1,\infty}(D_\refd;\mathbb{R}^d)}\Big].
\end{aligned}
\end{equation} 
In addition, we denote the concatinatination of two vectors or multi indices \(\balpha,\bbeta\) by \([\balpha;\bbeta]\). In this view, we also define
\(\bgamma\isdef[\bgamma_a;\bgamma_{\bf V}]\).

Then there holds, cp.\ \cite[Lemma 5]{HPS16},
\[
\|\partial^\balpha_{\bf z}(\psi_k\circ{\bf V})\|_{L^\infty(\Gamma^{M};L^\infty(D_\refd))}\leq|\balpha|!c_\psi\bigg(\frac{d}{\rho\log 2}\bigg)^{|\balpha|}
\bgamma_{\bf V}^\balpha\quad\text{for all }\balpha\in\mathbb{N}^M,
\]
where the power of a vector and a multi index has always to be interpreted as the product of the powers of each of the components, i.e.\
\[
{\bf x}^{\balpha}\isdef x_1^{\alpha_1}\cdots x_d^{\alpha_d},\quad{\bf x}\in\mathbb{R}^d,\balpha\in\mathbb{N}^d.
\]
Assuming, for the sake of simplicity, that the expectation \(\E[a]\) satisfies the same bound \eqref{eq:analyticEF}, we obtain 
a similar result for \(\E[a]\circ{\bf V}\).
Therefore, we can estimate the derivatives of the transported diffusion coefficient \(a\circ{\bf V}\) according to
\[
\Big\|\partial^{[\balpha_;\bbeta]}_{[{\bf y}; {\bf z}]}(a\circ{\bf V})\Big\|_{L^\infty(\Gamma^{M+N};L^\infty(D_\refd))}
\leq|\bbeta|!c_\psi\bigg(\frac{d}{\rho\log 2}\bigg)^{|\bbeta|}\bgamma^\bbeta\bigg(1+\sum_{k=1}^N\sqrt{\lambda_k}\bigg),\quad\balpha={\bf 0}
\]
and
\[
\Big\|\partial^{[\balpha_;\bbeta]}_{[{\bf y}; {\bf z}]}(a\circ{\bf V})\Big\|_{L^\infty(\Gamma^{M+N};L^\infty(D_\refd))}
\leq|\bbeta|!\sqrt{\lambda_k}c_\psi\bigg(\frac{d}{\rho\log 2}\bigg)^{|\bbeta|}\bgamma^\bbeta,\quad\alpha_k =1
\]
and \(\alpha_j=0\) for \(j\neq k\).
Since \(a\) is affine with respect to \({\bf y}\in\Gamma^N\), all higher order derivatives with respect to any \(y_k\) vanish.
Using the crude bound 
\[
\Big\|\partial^{[\balpha_;\bbeta]}_{[{\bf y}; {\bf z}]}(a\circ{\bf V})\Big\|_{L^\infty(\Gamma^{(N+M)};L^\infty(D_\refd))}
\leq (|\balpha|+|\bbeta|)! Cc^{|\balpha|+|\bbeta|}\bgamma^{[\balpha;\bbeta]}
\]
for some constants \(C,c>0\),
it is then easy to show that the derivatives of the transported diffusion coefficient 
\({\bf A}({\bf X},{\bf y},{\bf z})\) exhibit the same behavior, cf.\cite{HPS16}. 
Therefore, we can conclude in the same way as in \cite[Theorem 5]{HPS16} and
obtain the following result.
\begin{theorem}\label{thm:smoothcase} Let \(\hat{u}({\bf y},{\bf z})\) be the solution to \eqref{eq:varFormPara} and assume that the eigenfunctions \(\{\psi_k\}_k\) in the
Karhunen-Lo\`eve expansion of the diffusion coefficient are analytic in accordance with \eqref{eq:analyticEF}. Then, there exist constants
\(C,c>0\) such that
\[
\Big\|\partial^{[\balpha;\bbeta]}_{[{\bf y}; {\bf z}]}\hat{u}\Big\|_{L^\infty(\Gamma^{(N+M)};H^1_0(D_\refd))}
\leq (|\balpha|+|\bbeta|)! Cc^{|\balpha|+|\bbeta|}\bgamma^{[\balpha,\bbeta]}\quad\text{for all }\balpha\in\mathbb{N}^N,\bbeta\in\mathbb{N}^M.
\]
\end{theorem}
Regularity estimates of this form allow for sparse quadrature and collocation methods. As we have seen, they come at
the cost of high regularity requirements of the underlying data, see also \cite{GP16,HPS16,CNT16,HSSS18}. One possibility to bypass this drawback in the case of small rough perturbations
is considered in the following section.
\section{Rough diffusion coefficients}\label{sec:roughDiffusion}
In this section, we assume that the diffusion coefficient under consideration can be decomposed
into an analytic, deterministic and uniformly elliptic part \(a_s({\bf x})\) satisfying the bounds from 
\eqref{eq:uniElliptic} and a 
non-smooth centered random part \(a_r({\bf x,\omega})\) of small magnitude, i.e.
\begin{equation}\label{eq:DiffCoeff}
a({\bf x},\omega)=a_s({\bf x})+\varepsilon a_r({\bf x,\omega})\quad\text{for }0<\varepsilon\ll 1,
\end{equation}
where \(\|a_r({\bf x,\omega})\|_{L^\infty(\Omega;L^\infty(D(\omega)))}\leq 1\) and \(\E[a_r]({\bf x})\equiv 0\). 
Note that the results presented in this section also remain valid if \(a_s\) is also subjected to randomness.

We assume that \(a_r\) exhibits a Karhunen-Lo\`eve expansion of the form
\[
a_r({\bf x},{\bf y}) =\sum_{k=1}^{N}\sqrt{\lambda_k}\psi_k({\bf x})y_k,\quad{\bf y}\in\Gamma^N
\]
with the corresponding product measure \(\rho_X({\bf y})\d{\bf y}\). Moreover, we assume that the
random vector field is represented as in \eqref{eq:KLexpansions2}.

Next, we decompose the transported diffusion coefficient according to
\begin{equation}\label{eq:transportedCoefficient}
 {\bf A}({\bf X},{\bf y},{\bf z})={\bf A}_s({\bf X},{\bf z}) + \varepsilon {\bf A}_r({\bf X},{\bf y},{\bf z}),
\end{equation}
with 
\[
{\bf A}_s({\bf X},{\bf z})\isdef (a_s\circ{\bf V})({\bf X},{\bf z})({\bf V}'^\intercal{\bf V}')^{-1}({\bf X},{\bf z})\det{\bf V}'({\bf X},{\bf z})
\]
and
\[
{\bf A}_r({\bf X},{\bf y},{\bf z})\isdef (a_r\circ{\bf V})({\bf X},{\bf y},{\bf z})({\bf V}'^\intercal{\bf V}')^{-1}({\bf X},{\bf z})\det{\bf V}'({\bf X},{\bf z}).
\]
The independence and the centeredness, cf.\ \eqref{eq:centeredness}, of the families \(\{X_k\}_k\) and \(\{Y_k\}_k\) imply
\begin{align*}
\E[{\bf A}_r]({\bf X})&=\int_{\Gamma^N}\int_{\Gamma^M} 
\sum_{k=1}^{N}\sqrt{\lambda_k}(\psi_k\circ{\bf V})({\bf X},{\bf z})y_k
({\bf V}'^\intercal{\bf V}')^{-1}({\bf X},{\bf z})\det{\bf V}'({\bf X},{\bf z})
\rho_X({\bf y})\rho_Y({\bf z})\d{\bf z}\d{\bf y}\\
&=
\sum_{k=1}^{N}\sqrt{\lambda_k}\int_{\Gamma^M} (\psi_k\circ {\bf V})({\bf X},{\bf z})({\bf V}'^\intercal{\bf V}')^{-1}({\bf X},{\bf z})\det{\bf V}'({\bf X},{\bf z})\rho_Y({\bf z})\d{\bf z}\,\int_{\Gamma^N} y_k\rho_X({\bf y})\d{\bf y}\\
&= 0.
\end{align*}
Therefore, the rough part \({\bf A}_r({\bf X},{\bf y},{\bf z})\) of the transported diffusion coefficient remains centered. 

The pivotal idea is now, to linearize the coefficient-to-solution map in a vicinity of
the smooth part \(a_s\) of the diffusion coefficient and treat the perturbation \(a_r\) as a small systematic error.
To that end, we generalize the following theorem from \cite{HPS13a} for the case of a random 
diffusion coefficient and a random loading.
\begin{lemma}\label{lem:FrechetDerivative}
Let $\hat{u}_0({\bf z})\in H^1_0\big(D_\refd\big)$ be the solution to
\begin{equation}\label{eq:uZeroHat}
  -\div[{\bf A}_s({\bf z})\nabla\hat{u}_0({\bf z})]=f_\refd({\bf z})\ \text{in $D_\refd$},\quad
	\hat{u}_0({\bf z})=0\ \text{on $\partial D_\refd$}.
\end{equation}
Then, for \({\bf A}({\bf z})\isdef {\bf A}_s({\bf z})+\delta{\bf A}\), the mapping 
\[
  G\colon L^\infty(D_\refd;\mathbb{R}^{d\times d})\to H^1_0\big(D_\refd\big), \quad {\bf A}({\bf z})\mapsto G\big({\bf A}({\bf z})\big)\isdef\hat{u}({\bf z})
\]	
is Fr\'echet differentiable, where the derivative 
$\delta\hat{u}({\bf z}) = \delta\hat{u}[\delta{\bf A}]({\bf z})\in H_0^1(D_\refd)$ with respect to the direction 
$\delta{\bf A}\in L^\infty(D_\refd;\mathbb{R}^{d\times d})$ is given by
\begin{equation}	\label{eq:der-A}
  -\div\big({\bf A}_s({\bf z})\nabla \delta\hat{u}[\delta{\bf A}]({\bf z})\big)=\div\big(\delta{\bf A}\nabla\hat{u}_0({\bf z})\big)\ \text{in $D_\refd$},\quad
	\delta \hat{u}[\delta{\bf A}]({\bf z})=0\ \text{on $\partial D_\refd$}.
\end{equation}
\end{lemma}
\begin{proof} Let \(\varepsilon>0\) be sufficiently small such that 
\({\bf A}_s({\bf z})+\varepsilon\delta{\bf A}\) is still uniformly elliptic.
We consider the diffusion problems
\begin{align*}
-\div\big({\bf A}_s({\bf z})\nabla\hat{u}_0({\bf z})\big) 
&=f_\refd({\bf z})\text{ in }D_\refd,\quad u({\bf z})=0\text{ on }\partial D_\refd,\\
-\div\big(({\bf A}_s({\bf z})+\varepsilon\delta{\bf A})\nabla\hat{u}_\varepsilon({\bf z})\big) 
&=f_\refd({\bf z})\text{ in }D_\refd,\quad u_\varepsilon({\bf z})=0\text{ on }\partial D_\refd.
\end{align*}
Considering the difference of their respective weak formulations yields for \((\hat{u}_\varepsilon-\hat{u}_0)({\bf z})\in H^1_0(D_\refd)\) the equation
\begin{equation}\label{eq:coeffDiff}
\int_{D_\refd}{\bf A}_s({\bf z})\nabla(\hat{u}_\varepsilon-\hat{u}_0)({\bf z})\nabla v 
+ \varepsilon\delta{\bf A}\nabla\hat{u}_\varepsilon\nabla v\d{\bf X} = 0\quad
\text{for all }v\in H^1_0(D_\refd).
\end{equation}
Hence, the function \(\frac 1 \varepsilon (\hat{u}_\varepsilon-\hat{u}_0)({\bf z})\) satisfies the variational formulation
\[
\int_{D_\refd}{\bf A}_s({\bf z})\nabla\frac{(\hat{u}_\varepsilon-\hat{u}_0)({\bf z})}{\varepsilon}\nabla v\d{\bf X}
=-\int_{D_\refd}\delta{\bf A}\nabla\hat{u}_\varepsilon({\bf z})\nabla v\d{\bf X}
\quad\text{for all }v\in H^1_0(D_\refd).
\]
From the uniform ellipticity of \({\bf A}_s({\bf z})\) and the boundedness of \(\delta{\bf A}\) it is easy to derive that
\[
\|(\hat{u}_\varepsilon-\hat{u}_0)({\bf z})\|_{H^1_0(D_\refd)}\to 0\quad\text{as }\varepsilon\to 0.
\]
Therefore, we conclude that 
\[
\delta\hat{u}[\delta{\bf A}]({\bf z})\isdef\lim_{\varepsilon\to 0}\frac{(\hat{u}_\varepsilon-\hat{u}_0)({\bf z})}{\varepsilon}
\]
is the G\^ateaux derivative of \(\hat{u}({\bf z})\) in direction \(\delta{\bf A}\). It remains to show that \(\delta\hat{u}[\delta{\bf A}]({\bf z})\)
is also the Frech\'et derivative. Since  \(\delta\hat{u}[\delta{\bf A}]({\bf z})\) is obviously linear, the Frech\'et differentiability follows from
\begin{align*}
&\|({\hat u}_\varepsilon-\hat{u}_0)({\bf z})-\varepsilon\delta\hat{u}[\delta{\bf A}]({\bf z})\|_{H^1_0(D_\refd)}\\
&\qquad\leq C\sup_{v\in H^1_0(D_\refd)}\frac{1}{\|v\|_{H^1_0(D_\refd)}}\int_{D_\refd}{\bf A}_s({\bf z})\nabla ({\hat u}_\varepsilon-\hat{u}_0)({\bf z})-\varepsilon\delta\hat{u}[\delta{\bf A}]({\bf z})\nabla v\d{\bf X}\\
&\qquad = C\sup_{v\in H^1_0(D_\refd)}\frac{1}{\|v\|_{H^1_0(D_\refd)}}
\bigg(\int_{D_\refd}\big[\big({\bf A}_s({\bf z})+\varepsilon\delta{\bf A}\big)\nabla {\hat u}_\varepsilon({\bf z})-{\bf A}_s({\bf z})\nabla\hat{u}_0({\bf z})\big]\nabla v\d{\bf X}\\
&\qquad\qquad\qquad\qquad-\varepsilon\int_{D_\refd}\big[\delta{\bf A}\nabla\hat{u}_\varepsilon({\bf z})+{\bf A}_s({\bf z})
\nabla\delta\hat{u}[\delta{\bf A}]({\bf z})\big]\nabla v\d{\bf X}\bigg)\\
&\qquad= \varepsilon C\sup_{v\in H^1_0(D_\refd)}\frac{1}{\|v\|_{H^1_0(D_\refd)}}\int_{D_\refd}\delta{\bf A}\nabla(\hat{u}_\varepsilon-\hat{u}_0)({\bf z})\nabla v\d{\bf X}\\
&\qquad\leq \varepsilon C\|\delta{\bf A}\|_{L^\infty(D_\refd;\mathbb{R}^{d\times d})}\|(\hat{u}_\varepsilon-\hat{u}_0)({\bf z})\|_{H^1_0(D_\refd)}.
\end{align*}
Herein, the first inequality follows from the uniform ellipticity, i.e.\ \(C>0\) is the inverse of the ellipticity constant, while the second equality follows from \eqref{eq:der-A} and \eqref{eq:coeffDiff}.
\end{proof}

The proof of the previous lemma requires only the boundedness of the perturbation \(\delta{\bf A}\). 
Therefore, the following corollary is immediate.
\begin{corollary} The derivative \(\delta\hat{u}({\bf y},{\bf z})\isdef\delta\hat{u}[{\bf A}_r]({\bf y},{\bf z})\) is given by the diffusion problem
\[
  -\div\big({\bf A}_s({\bf z})\nabla \delta\hat{u}({\bf y},{\bf z})\big)=\div\big({\bf A}_r({\bf y},{\bf z})\nabla\hat{u}_0({\bf z})\big)\ \text{in $D_\refd$},\
	\delta \hat{u}({\bf y},{\bf z})=0\ \text{on $\partial D_\refd$}.
\]
Moreover, due to the linearity of the Frech\'et derivative, there holds
\begin{equation}\label{eq:derExpansion}
\delta\hat{u}({\bf y},{\bf z})=\sum_{k=1}^N\sqrt{\lambda_k}
\delta\hat{u}[(\psi_k\circ{\bf V})({\bf z})({\bf V}'^\intercal{\bf V}')^{-1}({\bf z})\det{\bf V}'({\bf z})]y_k.
\end{equation}
\end{corollary}
As a consequence of Lemma~\ref{lem:FrechetDerivative} and the subsequent corollary, 
we may expand $\hat{u}_\varepsilon({\bf y},{\bf z})$ into 
a first order Taylor expansion 
\begin{equation}	\label{eq:Taylor-A}
  \hat{u}_\varepsilon({\bf y},{\bf z}) =\hat{u}_0({\bf z}) + \varepsilon \delta \hat{u}[{\bf A}_r({\bf y},{\bf z})]({\bf y},{\bf z})
		+ \mathcal{O}(\varepsilon^2), 
\end{equation}
where \(\hat{u}_0({\bf z})\) is the solution to \eqref{eq:uZeroHat}.
Based on this expansion, the next theorem gives us expansions for the expectation and the variance of \(\hat{u}_\varepsilon\), see also \cite{HPS13a}.

\begin{theorem}		\label{thm:A}
For $\varepsilon > 0$ 
sufficiently small, there holds
\[
  \E[\hat{u}_\varepsilon] = \E[\hat{u}_0]+\mathcal{O}(\varepsilon^2)\quad\text{in $H^1_0(D_\refd)$}.
\]
Herein, $\hat{u}_0\in H^1_0(D_\refd)$ 
satisfies the diffusion problem \eqref{eq:uZeroHat}.
Moreover, there holds
\[
  \V[\hat{u}_\varepsilon] = \V[\hat{u}_0]+\mathcal{O}(\varepsilon^2)\quad\text{in $W^{1,1}_0(D_\refd)$}.
\]
\end{theorem}

\begin{proof}
To simplify notation, we abbreviate again $\delta\hat{u}({\bf y},{\bf z})\isdef\delta\hat{u}[{\bf A}_r]({\bf y},{\bf z})$.
Applying the Taylor expansion \eqref{eq:Taylor-A}, we arrive at
\[
\E[\hat{u}_\varepsilon] = \E[\hat{u}_0+\varepsilon\delta\hat{u}+\mathcal{O}(\varepsilon^2)]
=\E[\hat{u}_0]+\varepsilon\E[\delta\hat{u}]+\mathcal{O}(\varepsilon^2),
\]
by the linearity of the expectation.
Thus, we have to show that \(\E[\delta\hat{u}] = 0\). Employing \eqref{eq:derExpansion}, it holds
\begin{align*}
\E[\delta\hat{u}]
&=\int_{\Gamma^N}\int_{\Gamma^M}\delta\hat{u}({\bf y},{\bf z})\rho_X({\bf y})\rho_Y({\bf z})\d{\bf z}\d{\bf y}\\
&=\int_{\Gamma^N}\int_{\Gamma^M}\sum_{k=1}^N\sqrt{\lambda_k}
\delta\hat{u}[(\psi_k\circ{\bf V})({\bf z})({\bf V}'^\intercal{\bf V}')^{-1}({\bf z})\det{\bf V}'({\bf z})]y_k
\rho_X({\bf y})\rho_Y({\bf z})\d{\bf z}\d{\bf y}\\
&=\sum_{k=1}^N\sqrt{\lambda_k}\int_{\Gamma^M}
\delta\hat{u}[(\psi_k\circ{\bf V})({\bf z})({\bf V}'^\intercal{\bf V}')^{-1}({\bf z})\det{\bf V}'({\bf z})]
\rho_Y({\bf z})\d{\bf z}\int_{\Gamma^N}y_k\rho_X({\bf y})\d{\bf y}\\
&=0
\end{align*}
due to the centeredness of the parameters \(y_k\), see \eqref{eq:centeredness}. 

Employing the result for the expectation, we obtain for the variance
\begin{align*}
\V[\hat{u}_\varepsilon]&=\int_{\Gamma^N}\int_{\Gamma^M}(\hat{u}_\varepsilon({\bf y},{\bf z})-\E[\hat{u}_\varepsilon])^2\rho_X({\bf y})\rho_Y({\bf z})\d{\bf z}\d{\bf y}\\
&=\int_{\Gamma^N}\int_{\Gamma^M}\big((\hat{u}_0({\bf z})-\E[u_0])+\varepsilon\delta\hat{u}({\bf y},{\bf z})+\mathcal{O}(\varepsilon^2)\big)^2\rho_X({\bf y})\rho_Y({\bf z})\d{\bf z}\d{\bf y}\\
&=\V[\hat{u}_0]+2\int_{\Gamma^N}\int_{\Gamma^M}(\hat{u}_0({\bf z})-\E[u_0])\big(\varepsilon\delta\hat{u}({\bf y},{\bf z})+\mathcal{O}(\varepsilon^2)\big)\rho_X({\bf y})\rho_Y({\bf z})\d{\bf z}\d{\bf y}\\
&\phantom{=} + \varepsilon^2\int_{\Gamma^N}\int_{\Gamma^M}
\big(\delta\hat{u}({\bf y},{\bf z})+\mathcal{O}(\varepsilon)\big)^2\rho_X({\bf y})\rho_Y({\bf z})\d{\bf z}\d{\bf y}.
\end{align*}
Since the last term is bounded with respect to the \(W^{1,1}_0(D_\refd)\)-norm, see e.g.\ \cite{HPS16e}, it remains to show 
that the \(\varepsilon\) dependent part in the second term vanishes. Again making use of \eqref{eq:derExpansion},
this can be seen as follows
\begin{align*}
&2\varepsilon\int_{\Gamma^N}\int_{\Gamma^M}(\hat{u}_0({\bf z})-\E[u_0])\delta\hat{u}({\bf y},{\bf z})\rho_X({\bf y})\rho_Y({\bf z})\d{\bf z}\d{\bf y}\\
&\qquad = 
\int_{\Gamma^N}\int_{\Gamma^M}(\hat{u}_0({\bf z})-\E[u_0])\sum_{k=1}^N\sqrt{\lambda_k}
\delta\hat{u}[(\psi_k\circ{\bf V})({\bf z})({\bf V}'^\intercal{\bf V}')^{-1}({\bf z})\det{\bf V}'({\bf z})]y_k\rho_X({\bf y})\rho_Y({\bf z})\d{\bf z}\d{\bf y}\\
&\qquad = 
\sum_{k=1}^N\sqrt{\lambda_k}\int_{\Gamma^M}(\hat{u}_0({\bf z})-\E[u_0])
\delta\hat{u}[(\psi_k\circ{\bf V})({\bf z})({\bf V}'^\intercal{\bf V}')^{-1}({\bf z})\det{\bf V}'({\bf z})]\rho_Y({\bf z})\d{\bf z}\int_{\Gamma^N}y_k\rho_X({\bf y})\d{\bf y}\\
&\qquad=0,
\end{align*}
where we again exploit the centeredness of the coordinates \(y_k\). This completes the proof.
\end{proof}

\begin{remark}
As the proof indicates, the next order correction term for the variance would be \(\E[\delta\hat{u}^2]\),
which would then lead to the expansion
\[
\V[\hat{u}_\varepsilon]=\V[\hat{u}_0] + \varepsilon^2\E[\delta\hat{u}^2]+\mathcal{O}(\varepsilon^3). 
\]
Here, in contrast to e.g.\ \cite{HPS13a}, the term \(\V[\hat{u}_0]\) does not vanish, since \(\hat{u}_0\)
is still a non-constant random field due to its dependency on the random domain.
\end{remark}

\begin{remark}
If the diffusion coefficient is already given in material coordinates, the summands in expression \eqref{eq:transportedCoefficient} simplify towards
\[
{\bf A}_s({\bf X},{\bf z})\isdef a_s({\bf X})({\bf V}'^\intercal{\bf V}')^{-1}({\bf X},{\bf z})\det{\bf V}'({\bf X},{\bf z})
\]
and
\[
{\bf A}_r({\bf X},{\bf y},{\bf z})\isdef a_r({\bf X},{\bf y})({\bf V}'^\intercal{\bf V}')^{-1}({\bf X},{\bf z})\det{\bf V}'({\bf X},{\bf z}).
\]
In this case, it is also easy to see that \(\E[{\bf A}_r]=0\).
Consequently, Theorem~\ref{thm:A} remains valid in this case. 
Moreover, it is straightforward to apply the previous derivation to diffusion problems with anisotropic
coefficients, i.e.\ \(a({\bf x},\omega)\) is a matrix-valued function, cf.\ \cite{HPSc17}.
\end{remark}

Theorem~\ref{thm:A} tells us that, given the boundedness in the rough part \(a_r\), we obtain a quadratic approximation of the solution's expectation
and variance in terms of the perturbation's amplitude. A similar result holds for output functionals of the solution
\begin{corollary}Let \(F\colon H^1_0(D_\refd)\to\mathbb{R}\) be a continuous and linear functional. Then, for $\varepsilon > 0$ 
sufficiently small, there holds
\[
  \E[F(\hat{u}_\varepsilon)] = \E[F(\hat{u}_0)]+\mathcal{O}(\varepsilon^2)\quad\text{in $H^1_0(D_\refd)$}.
\]
Herein, $\hat{u}_0\in H^1_0(D_\refd)$ 
satisfies the diffusion problem \eqref{eq:uZeroHat}.
Moreover, there holds
\[
  \V[F(\hat{u}_\varepsilon)] = \V[F(\hat{u}_0)]+\mathcal{O}(\varepsilon^2)\quad\text{in $W^{1,1}_0(D_\refd)$}.
\]
\end{corollary}
\begin{proof}
By exploiting the continuity and the linearity of \(F\), the proof can be conducted similarly to the proof of Theorem~\ref{thm:A}.
\end{proof}

\section{Numerical realization of the domain mapping method}\label{sec:numericalRealization}
\begin{figure}[htb]
\begin{center}
\scalebox{0.75}{
\begin{tikzpicture}[
    scale=.4,
    axis/.style={thick, ->, >=stealth'},
    important line/.style={thick},
    every node/.style={color=black}
    ]
\draw (-.5,-1)node{{\large $(0,0)$}};
\draw (8,-1)node{{\large $(1,0)$}};
\draw (-.5,9)node{{\large $(0,1)$}};
\draw (0,0)--(8,0);
\draw (0,0)--(0,8);
\draw (8,0)--(0,8);
\draw (4,0)--(0,4);
\draw (6,0)--(0,6);
\draw (2,0)--(0,2);
\draw (2,0)--(2,6);
\draw (0,2)--(6,2);
\draw (0,4)--(4,4);
\draw (4,0)--(4,4);
\draw (6,0)--(6,2);
\draw (0,6)--(2,6);
\draw (2,2)node(N1){};
\draw (14,6)node(N2){};
\draw (16.5,3)node{\includegraphics[scale=.2, trim= 300 140 300 140, clip]{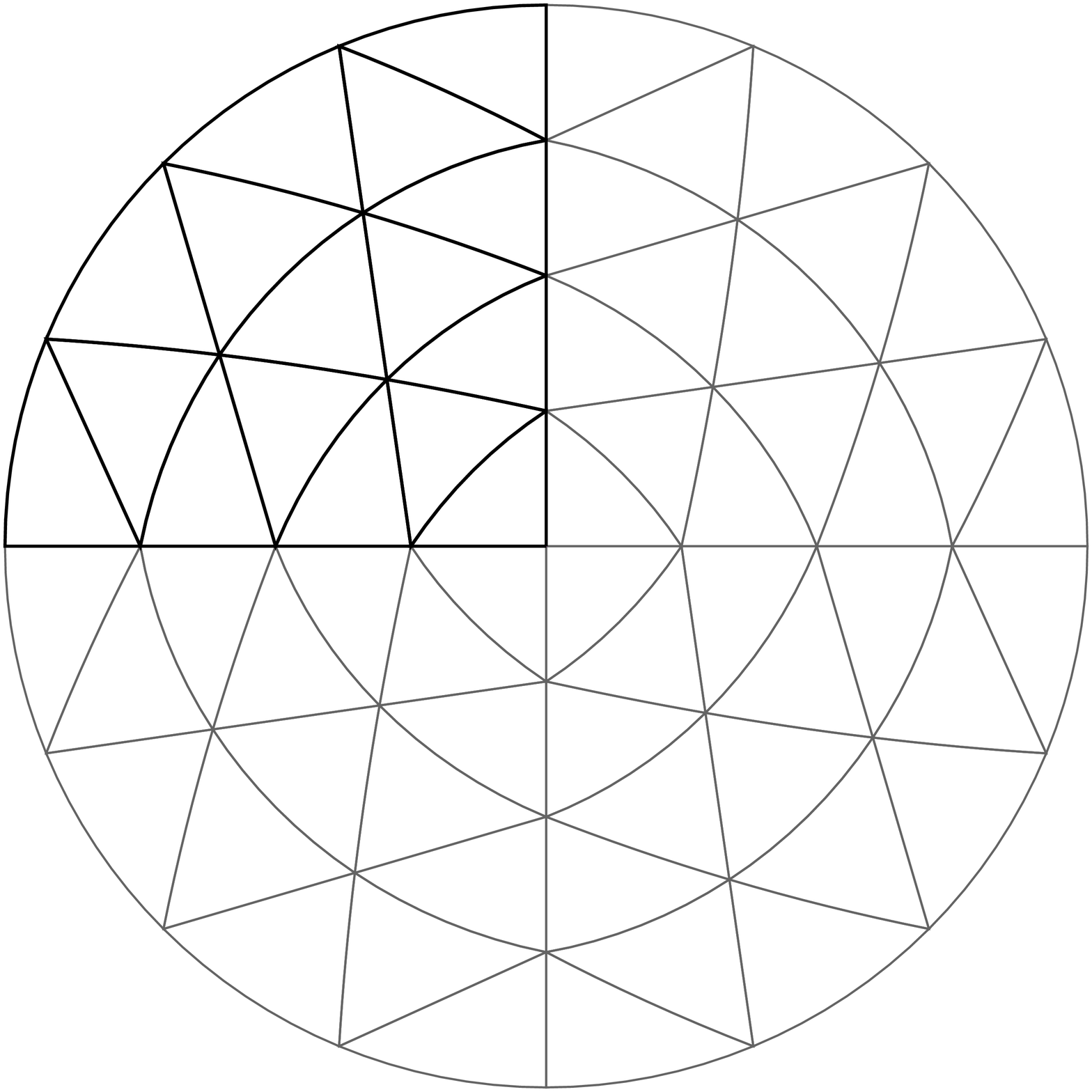}};
\draw (32,3)node{\includegraphics[scale=.2, trim= 300 140 300 140, clip]{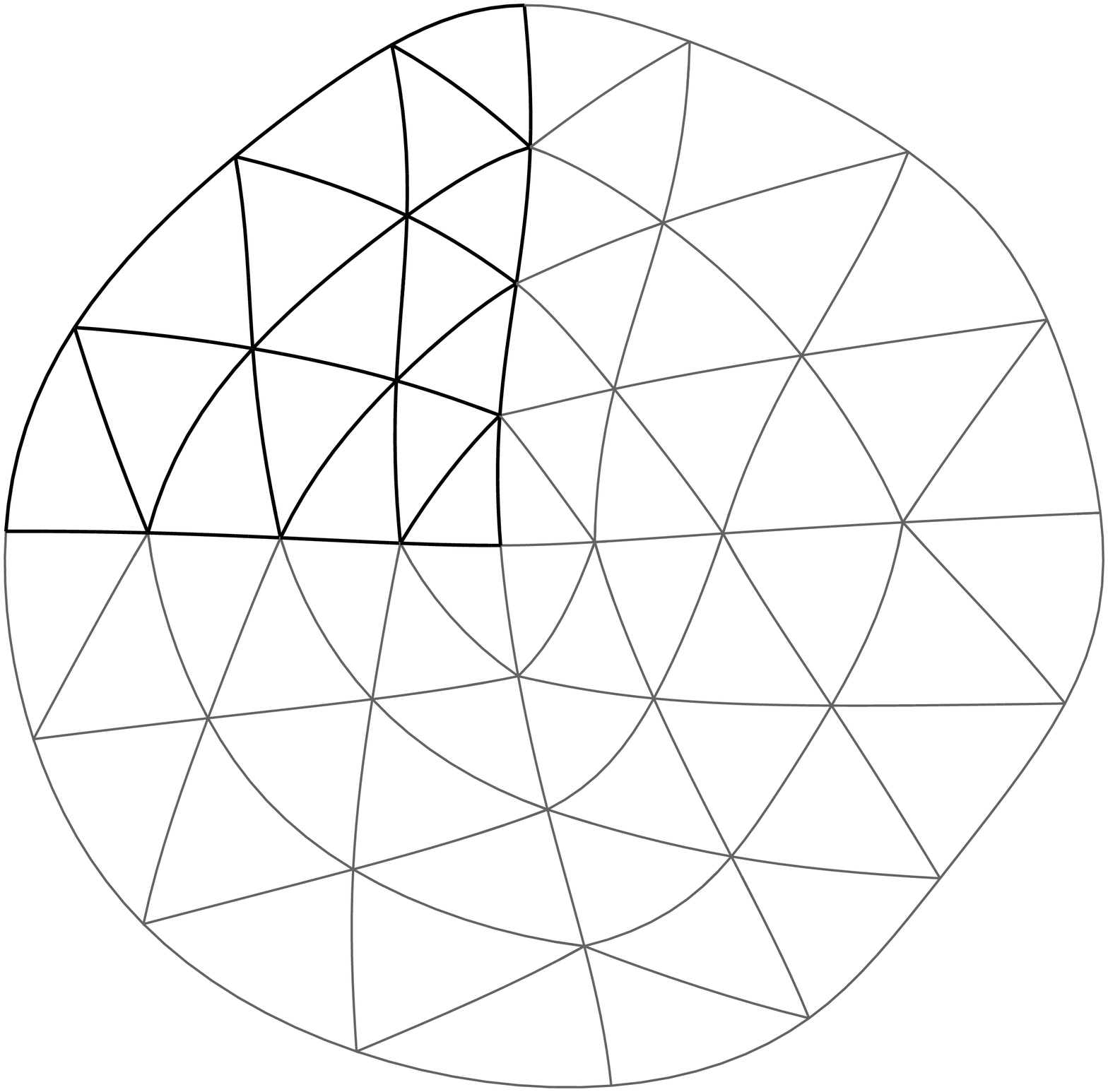}};
\draw (7.9,6.5)node(N3){\large${\boldsymbol\kappa}_j$};
\path[->,line width=1pt]  (N1) edge [bend left] (N2);
\draw (14.5,6)node(N4){};
\draw (31.1,4.8)node(N5){};
\path[->,line width=1pt]  (N4) edge [bend left] (N5);
\draw (23.2,8.5)node(N3){\large${\bf V}({\bf X},{\bf y}_i)$};
\end{tikzpicture}}
\caption{\label{fig:parametrization} Construction of parametric finite elements and the action of the random vector field.}
\end{center}
\end{figure}
In this section, we describe the numerical realization of the domain mapping method
and, in particular, the discretization of the random vector field.
In addition, we give a brief description of how the diffusion coefficient, which is 
given in spatial coordinates, can be represented such that it is feasible within the
domain mapping method.
For the sake of completeness, we start by introducing the parametric 
finite element method which is used in the numerical examples. Nevertheless, 
we emphasize that the construction presented in Section~\ref{subsec:randField} is
also viable for other finite element discretizations based on nodal basis functions. 
\subsection{Finite element approximation}
For the spatial discretization, we employ linear 
(iso-) parametric finite elements, see \cite{B,Brenner,lenoir}.
To that end, the domain $D_\refd$ shall be given by a 
collection of simplicial smooth \emph{patches}. More 
precisely, let $\triangle$ denote the reference simplex 
in $\mathbb{R}^d$. Then, the domain 
$D_\refd$ shall be partitioned into $K$ patches
\begin{equation}\label{eq:parametrization}
  \overline{D_\refd} = \bigcup_{j=1}^K \tau_{0,j},
  	\quad \tau_{0,j} = \boldsymbol\kappa_j(\triangle), 
	\quad j = 1,2,\ldots,K.
\end{equation}

The intersection \(\tau_{0,j}\cap \tau_{0,j'}\), \(j\neq j'\), of
any two patches \(\tau_{0,j}\) and \(\tau_{0,j'}\) is supposed 
to be either empty or a common lower dimensional face.

Based on this construction, it is straightforward to introduce a hierarchical
mesh on \(D_\refd\), which is feasible for a geometric multigrid solver: 
A mesh of level \(\ell\) on \(D_\refd\) is obtained by 
regular subdivisions of depth \(\ell\) of the reference simplex 
into \(2^{\ell d}\) sub-simplices. Then, by mapping this mesh via the
parametrizations \(\boldsymbol\kappa_j\), \(j=1,\ldots,K\), we obtain
 \(K2^{\ell d}\) elements $\{\tau_{\ell,i}\}_{i}$ for \(D_\refd\),
 see the first
mapping in Figure~\ref{fig:parametrization} for a visualization of this situation.
 
 In order to guarantee that the triangulation
\(\mathcal{T}_\ell\isdef\{\tau_{\ell,i}\}_i\) on level \(\ell\) results in
a regular mesh for \(D_\refd\), the parametrizations 
\(\{ \boldsymbol\kappa_j\}_j\) are supposed to be \(C^0\) 
compatible in the following sense: there exists a bijective, 
affine mapping \(\boldsymbol\Xi:\triangle\to\triangle\)
such that for all \({\bf X} = \boldsymbol\kappa_i({\bf s})\)
on a common interface of \(\tau_{0,j}\) and \(\tau_{0,j'}\) 
it holds that \(\boldsymbol\kappa_j({\bf s})
= (\boldsymbol\kappa_{j'}\circ\boldsymbol\Xi)({\bf s})\).
Thus, the diffeomorphisms \(\boldsymbol\kappa_j\) and
\(\boldsymbol\kappa_{j'}\) coincide at the common interface except
for orientation. 

Finally, we define the piecewise finite element ansatz functions 
by lifting the Lagrangian finite elements from \(\triangle\) 
to the domain \(D_\refd\) by using the parametrizations \(\boldsymbol\kappa_j\).
To that end, let \(\Phi_\ell=\{\varphi_{\ell,i}:i\in\mathcal{I}_\ell\}\), 
where \(\mathcal{I}_\ell\) is a suitable index set,
denote the Lagrangian piecewise linear basis functions
on the \(\ell\)-th subdivision \(\triangle_\ell\)
of the reference simplex.
Then, we obtain the finite element space
\[
  {V}_{\triangle,\ell}= \operatorname{span}\{\varphi_{\ell,j}:j\in\mathcal{I}_\ell\}
	= \{ u\in C(\triangle): u|_\tau\in\Pi_1\ \text{for all}\ \tau\in\triangle_\ell\}
\]
with $\dim{V}_{\triangle,\ell}\eqsim 2^{\ell d}$ and \(\Pi_1\) 
denoting the space of linear polynomials.
Continuous basis functions whose support overlaps 
with several patches are obtained by gluing across 
patch boundaries, using the $C^0$ inter-patch compatibility. 
This yields a nested sequence of finite element 
spaces
\[
V_{\refd,\ell}\isdef\{v\in C(D_\refd): v|_{{\boldsymbol\kappa}_j(\triangle)}=\varphi\circ{\boldsymbol\kappa}^{-1}_j,
\varphi\in{V}_{\triangle,\ell},\ j=1,\ldots,K\}\subset H^1(D_\refd)
\]
with \(\dim V_{\refd,\ell}\eqsim K2^{\ell d}\). 
It is well known that 
the spaces \(V_{\refd,\ell}\) satisfy the following approximation result. It holds
\begin{equation}	\label{eq:approx}
 \inf_{v_\ell\in V_{\refd,\ell}}\|u-v_\ell\|_{H^1(D_\refd)}
 	\leq c h_\ell\|u\|_{H^2(D_\refd)}
		\quad u\in H^2(D_\refd),
\end{equation}
for some constant \(c>0\), see e.g.\ \cite{B,Brenner}. This result remains valid if we replace the parametrizations
\(\{\boldsymbol\kappa_j\}_j\) by their piecewise affine approximation, i.e. when we replace
the curved edges by planar ones,
 see again \cite{B,Brenner}. In this view, one might also directly start from a polygonal approximation of \(D_\refd\).
 \subsection{Discretization of the random vector field}\label{subsec:randField}
 Next, we explain how the random vector field can be approximated numerically. For the sake of a less cumbersome
 notation, we present the construction only for \(d=2\) and emphasize that the construction for \(d=3\) can
 be performed in complete analogy. 
 
 Let the random vector field be given by its expectation \(\E[{\bf V}]({\bf X})\) 
 and its covariance function
 \(\Cov[{\bf V}]({\bf X},{\bf X}')=[\Cov_{i,j}[{\bf V}]({\bf X},{\bf X}')]_{i,j=1,2}\).
Moreover, let \(\{{\bf X}_i\}_{i=1}^n\subset D_\refd\) be the centers of the nodal linear finite element basis
\(\{\varphi_1,\ldots,\varphi_n\}\), i.e.\ \(\varphi_i({\bf X}_j)=\delta_{i,j}\), where \(n=\operatorname{dim}(V_{\refd,\ell})\).
Then, we can represent the expectation by its finite element interpolant
\[
\E[{\bf V}]({\bf X})\approx\sum_{i=1}^n \E[{\bf V}]({\bf X}_i)\varphi_i({\bf X})
\]
and in complete analogy
\[
\Cov[{\bf V}]({\bf X},{\bf X}')\approx\sum_{i,j=1}^n\Cov[{\bf V}]({\bf X}_i,{\bf X}_j)\varphi_i({\bf X})\varphi_j({\bf X}').
\]
Now, in order to determine the Karhunen-Lo\`eve expansion of \({\bf V}\), 
we have to solve the operator eigenvalue problem
\[
\int_{D_\refd}\Cov[{\bf V}]({\bf X},{\bf X}'){\boldsymbol\varphi}({\bf X}')\d{\bf X}'=\mu{\boldsymbol\varphi}({\bf X}).
\]
Thus, by replacing \(\Cov[{\bf V}]\) by its finite element interpolant and testing with respect to the basis functions
\(\big\{[\varphi_i,0]^\intercal_{i},[0,\varphi_i]^\intercal_{i}\big\}_{i=1}^n\), 
we end up with the generalized algebraic eigenvalue problem
\begin{equation}\label{eq:algEig}
\begin{bmatrix}{\bf M} & \\ & {\bf M}\end{bmatrix}
{\bf C}
\begin{bmatrix}{\bf M} & \\ & {\bf M}\end{bmatrix}
{\bf v}=\mu\begin{bmatrix}{\bf M} & \\ & {\bf M}\end{bmatrix}{\bf v},\quad{\bf v}\in\mathbb{R}^{2n}.
\end{equation}
Herein,
\[
{\bf C}\isdef\begin{bmatrix}
 \big[\Cov_{1,1}[{\bf V}]({\bf X}_i,{\bf X}_j)\big]_{i,j=1}^n &\big[\Cov_{1,2}[{\bf V}]({\bf X}_i,{\bf X}_j)\big]_{i,j=1}^n\\
\big[\Cov_{2,1}[{\bf V}]({\bf X}_i,{\bf X}_j)\big]_{i,j=1}^n &\big[\Cov_{2,2}[{\bf V}]({\bf X}_i,{\bf X}_j)\big]_{i,j=1}^n
\end{bmatrix}\in\mathbb{R}^{2n\times 2n}
\]
is the covariance function evaluated in all combinations of grid points and
\[
{\bf M}\isdef[m_{i,j}]_{i,j=1}^n\in\mathbb{R}^{n\times n}\quad\text{with}\quad m_{i,j}=
\int_{D_\refd}\varphi_i\varphi_j\d{\bf X}
\]
denotes the finite element mass matrix.

The algebraic eigenvalue problem \eqref{eq:algEig} can now be efficiently solved by means of the pivoted Cholesky decomposition 
as follows.
Let \({\bf C}\approx{\bf L}{\bf L}^\intercal\) with \({\bf L}\in\mathbb{R}^{2n\times M}\) (\(M\ll n\))
be the low rank approximation provided by the 
pivoted Cholesky decomposition of \({\bf C}\) as described in, e.g.\ \cite{HPS14}.
Then, we approximate the eigenvalue problem \eqref{eq:algEig} by
\begin{equation}\label{eq:approxEig}
\begin{bmatrix}{\bf M} & \\ & {\bf M}\end{bmatrix}
{\bf LL}^\intercal
\begin{bmatrix}{\bf M} & \\ & {\bf M}\end{bmatrix}
{\bf v}=\mu\begin{bmatrix}{\bf M} & \\ & {\bf M}\end{bmatrix}{\bf v},\quad{\bf v}\in\mathbb{R}^{2n}.
\end{equation}
This eigenvalue problem is equivalent to the usually much smaller eigenvalue problem
\begin{equation}\label{eq:redEig}
{\bf L}^\intercal\begin{bmatrix}{\bf M} & \\ & {\bf M}\end{bmatrix}{\bf L}\tilde{\bf v}=\mu\tilde{\bf v},
\quad\tilde{\bf v}\in\mathbb{R}^M.
\end{equation}
In particular, if \(\tilde{\bf v}_i\) is an eigenvector of \eqref{eq:redEig} with eigenvalue \(\mu_i\), then \({\bf v}_i\isdef{\bf L}\tilde{\bf v}_i\) is an eigenvector of \eqref{eq:approxEig} with eigenvalue \(\mu_i\). Moreover, there holds
\[
{\bf v}_i^\intercal\begin{bmatrix}{\bf M} & \\ & {\bf M}\end{bmatrix}{\bf v}_j=\mu_i\delta_{i,j}.
\]
\begin{remark}
The cost for computing the pivoted Cholesky decomposition is \(\mathcal{O}(2nM^2)\) and, since all entries of
\({\bf C}\) can be computed on the fly without the need of storing the entire matrix \({\bf C}\), 
the storage cost is \(\mathcal{O}(2nM)\). Moreover,
the small eigenvalue problem \eqref{eq:redEig} can be solved with cost \(\mathcal{O}(M^3)\). Thus, since usually \(M\ll n\) 
the overall cost for computing the Karhunen-Lo\`eve expansion of \({\bf V}\) by the suggested approach 
is also \(\mathcal{O}(2nM^2)\) in total.
\end{remark}

By the presented construction, we obtain a piecewise affine approximation of the random vector field \({\bf V}\), see
the second mapping in Figure~\ref{fig:parametrization}.
Note that the uniformity condition \eqref{eq:unif} guarantees that the functional determinant \(\det{\bf V}'\)
has a constant sign, see e.g.\ \cite{HPS16}. Thus, without loss of generality, we may assume \(\det{\bf V}'>0\). 
For a sufficiently small mesh width \(h>0\), this property carries over to its piecewise affine approximation. 
Therefore, each random realization of the piecewise affine approximation maps a mesh onto a mesh 
by simply moving the mesh points and keeping the topology,
i.e. the sets of point indices that make up an element, fixed.  

As a consequence, the realizations of the solution to \eqref{eq:varForm} can either be directly 
computed on the reference domain
\(D_\refd\) or on its image \({\bf V}(D_\refd,\omega)\). 
In particular, there holds a similar approximation result to \eqref{eq:approx} for the mapped finite element space on 
\({\bf V}(D_\refd,\omega)\), see e.g. \cite{B,HPS16}.
Finally, we remark that in case that the solution \(u\) has been computed on \({\bf V}(D_\refd,\omega)\), 
the corresponding solution \(\hat{u}\) on \(D_\refd\) can be easily retrieved
by assigning the computed node values of the solution to the respective mesh points in \(D_\refd\), 
cf.\ \eqref{eq:solutionIdentity}.

\subsection{Discretization of the diffusion coefficient}
Since the diffusion coefficient is represented in spatial coordinates, 
we have to provide an efficient means to evaluate it for each particular 
realization of the domain \({\bf V}(D_\refd,\omega)\).
To that end, we assume that the hold-all \(\mathcal{D}\) 
can be subdivided by a Cartesian grid, for example,
the hold all can be chosen as a rectangle for \(d=2\) 
and as a cuboid for \(d=3\). Then, after introducing
a uniform grid for \(\mathcal{D}\), 
we can perform the computation of the Karhunen-Lo\`eve expansion exactly
as for the random vector field.

Now, in order to evaluate \(a({\bf x},\omega)\) at a certain point
\({\bf V}(D_\refd,\omega)\), we only need to retrieve the containing grid
cell in \(\mathcal{D}\) which is, due to the simple structure of the mesh
on \(\mathcal{D}\) an \(\mathcal{O}(1)\) operation.

\section{Numerical results}\label{sec:numres}
In this section, we present a numerical example to validate the
presented approach. For the sake of computation times, we consider
only an example in two spatial dimensions.
The reference domain is given by the unit disc, i.e.\
\[
D_\refd\isdef\{{\bf X}\in\mathbb{R}^2: \|{\bf X}\|_2 < 1\}.
\]
The random vector field is represented via its expectation
and covariance function according to
\begin{equation}\label{eq:covVField}
\E[{\bf V}]({\bf X})={\bf X},\quad \Cov[{\bf V}]({\bf X},{\bf X}')=
\frac{1}{1000}
\begin{bmatrix}
5\exp(-2\|{\bf X}-{\bf X}'\|_2^2) & \exp(-0.1\|2{\bf X}-{\bf X}'\|_2^2)\\
\exp(-0.1\|{\bf X}-2{\bf X}'\|_2^2) & 5\exp(-0.5\|{\bf X}-{\bf X}'\|_2^2)
\end{bmatrix}.
\end{equation}
Moreover, we assume that the random variables in the corresponding
Karhunen-Lo\`eve expansion are independent and uniformly distributed on
\([-\sqrt{3},\sqrt{3}]\), i.e.\ they have normalized variance.

We aim at computing the expectation and the variance of the solution to
\[
-\div\big(a_\varepsilon(\omega)\nabla u_\varepsilon(\omega)\big)=1\quad\text{in }D(\omega),
\quad u_\varepsilon(\omega)=0\quad\text{on }\partial D(\omega).
\]

\begin{figure}[htb]
\begin{center}
\includegraphics[scale=.5, trim= 100 250 100 250, clip]{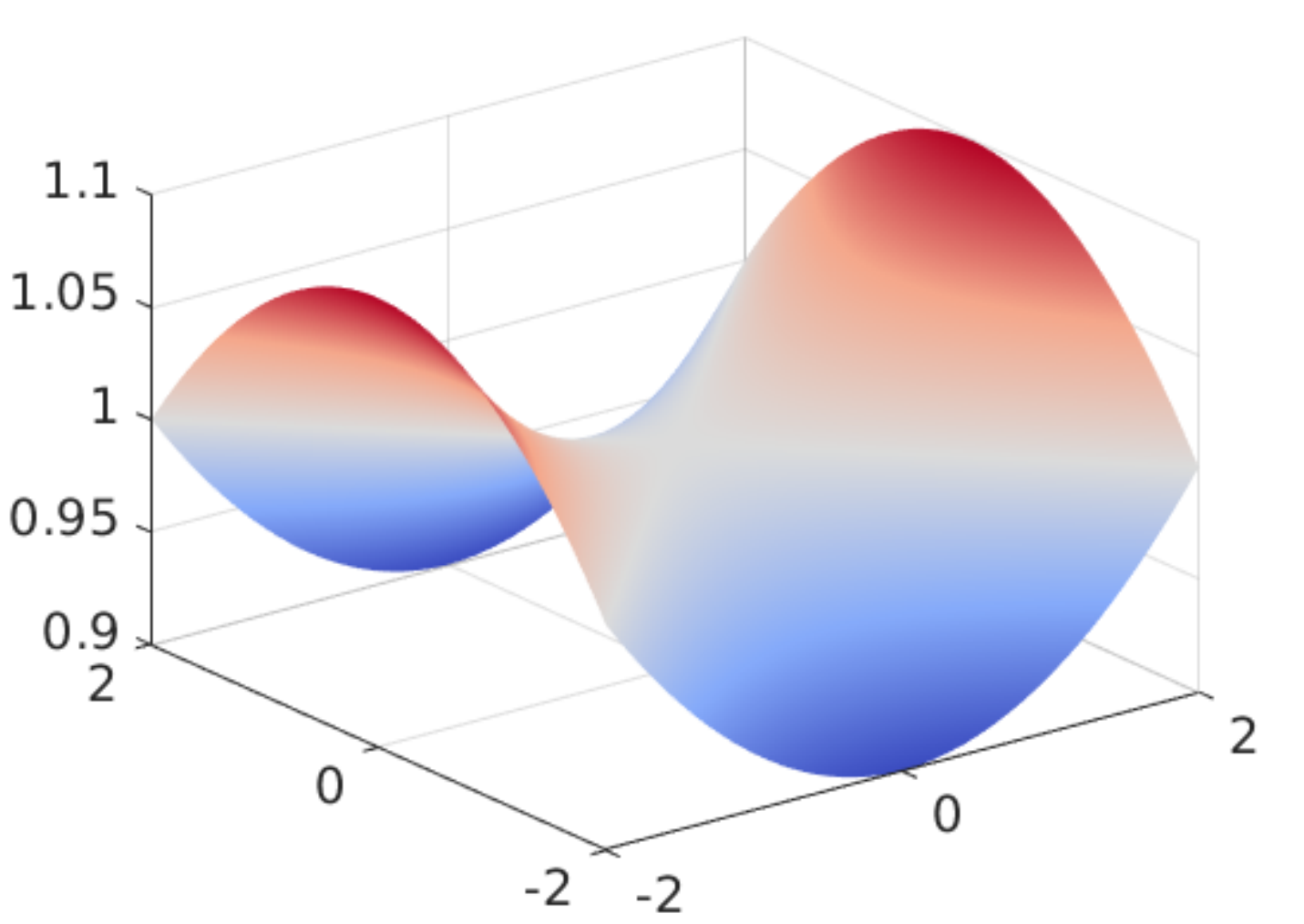}\hfill
\includegraphics[scale=.5, trim= 100 250 100 250, clip]{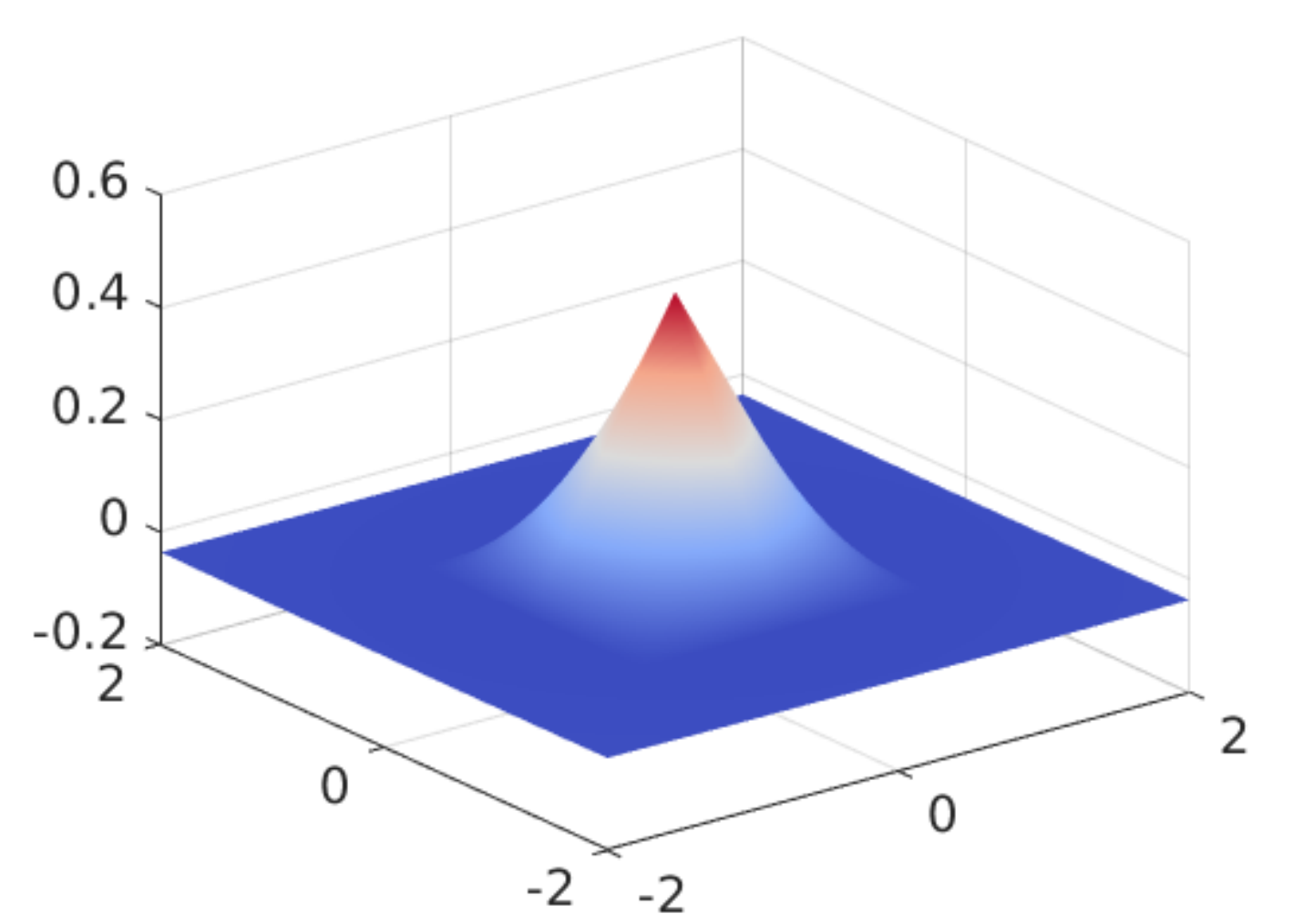}
\caption{\label{fig:meanandef}Expectation (left) and first eigenfunction (right) of the 
random diffusion coefficient.}
\end{center}
\end{figure}

Herein, the diffusion coefficient is also defined via its expectation and covariance as
\[
\E[a_\varepsilon]({\bf x}) = 1 + \frac{x_1^2 - x_2^2}{40},\quad\Cov[a_\varepsilon]({\bf x},{\bf x}')
=\frac{\varepsilon^2}{100}\bigg[2\exp\bigg(\frac{-\|{\bf x}-{\bf x}'\|_2^2}{32}\bigg)+
9g({\bf x})g({\bf x}')\bigg],\quad\varepsilon\geq0
\]
where 
\[
g({\bf x})\isdef\max\{0,1-|x_1|\}\cdot\max\{0,1-|x_2|\}
\]
denotes the tensor product hat-function. 
Again, we assume that the random variables in the corresponding
Karhunen-Lo\`eve expansion are independent and uniformly distributed on
\([-\sqrt{3},\sqrt{3}]\).

We remark that the tensor product hat function \(g({\bf x})\) is only Lipschitz continuous. 
Therefore, \(g\circ{\bf V}\) and, consequently, \(a\circ{\bf V}\) are non-smooth functions. 
See the right panel in Figure~\ref{fig:meanandef} for a visualization of
the first eigenfunction in the diffusion coefficient's Karhunen-Lo\`eve expansion. 
Consequently, Theorem~\ref{thm:smoothcase}
cannot be applied here and we resort to the perturbation approach.
To that end, we set
\[
a_\varepsilon({\bf x},\omega)=a_s({\bf x})+\varepsilon a_r({\bf x},\omega)\isdef
\E[a]({\bf x})+\varepsilon\sum_{k=1}^N\sqrt{\lambda_k}\psi_k({\bf x})X_k(\omega),
\]
where the latter refers to the truncated Karhunen-Lo\`eve expansion of \(a_\varepsilon({\bf x},\omega)\).

The computation of the finite element approximations and of the 
 Karhunen-Lo\`eve expansion for the random vector field are carried out on a mesh for \(D_\refd\)
with mesh width \(h=2^{-6}\). The Karhunen-Lo\`eve expansion is truncated 
such that the truncation error is smaller than \(10^{-2}\), see e.g.\ \cite{GP16}.
This results in \(M=52\) parameters. The Karhunen-Lo\`eve
expansion of the diffusion coefficient is computed on the hold-all domain \(\mathcal{D}=[-2,2]^2\), 
which is discretized by a quadrilateral mesh of size \(h=4\cdot10^{-3}\). 
The coefficient's Karhunen-Lo\`eve expansion is also truncated such that the error is smaller
than \(10^{-2}\). This results in \(N=10\) additional parameters. 

Hence, in terms of parameters, we are facing here a 62 dimensional quadrature problem with a non-smooth integrand. 
This means that we cannot use sophisticated sparse- or quasi-Monte Carlo quadrature methods to obtain suitable 
reference solutions. Instead, we employ the plain vanilla Monte Carlo quadrature with a huge amount, 
i.e. \(10^8\), of samples. Each reference solution
is then calculated by averaging five runs of the Monte Carlo simulation, 
resulting in \(5\cdot 10^8\) samples in total.\footnote{The computations have been carried 
out on the Euler cluster managed by the Scientific IT Services at the ETH Zurich, see
\texttt{https://sis.id.ethz.ch/hpc}, with up to 1400 cores.}

\begin{figure}[htb]
\begin{center}
\begin{tikzpicture}
\draw (0,6) node{
\includegraphics[scale=.5, trim= 100 250 100 250, clip]{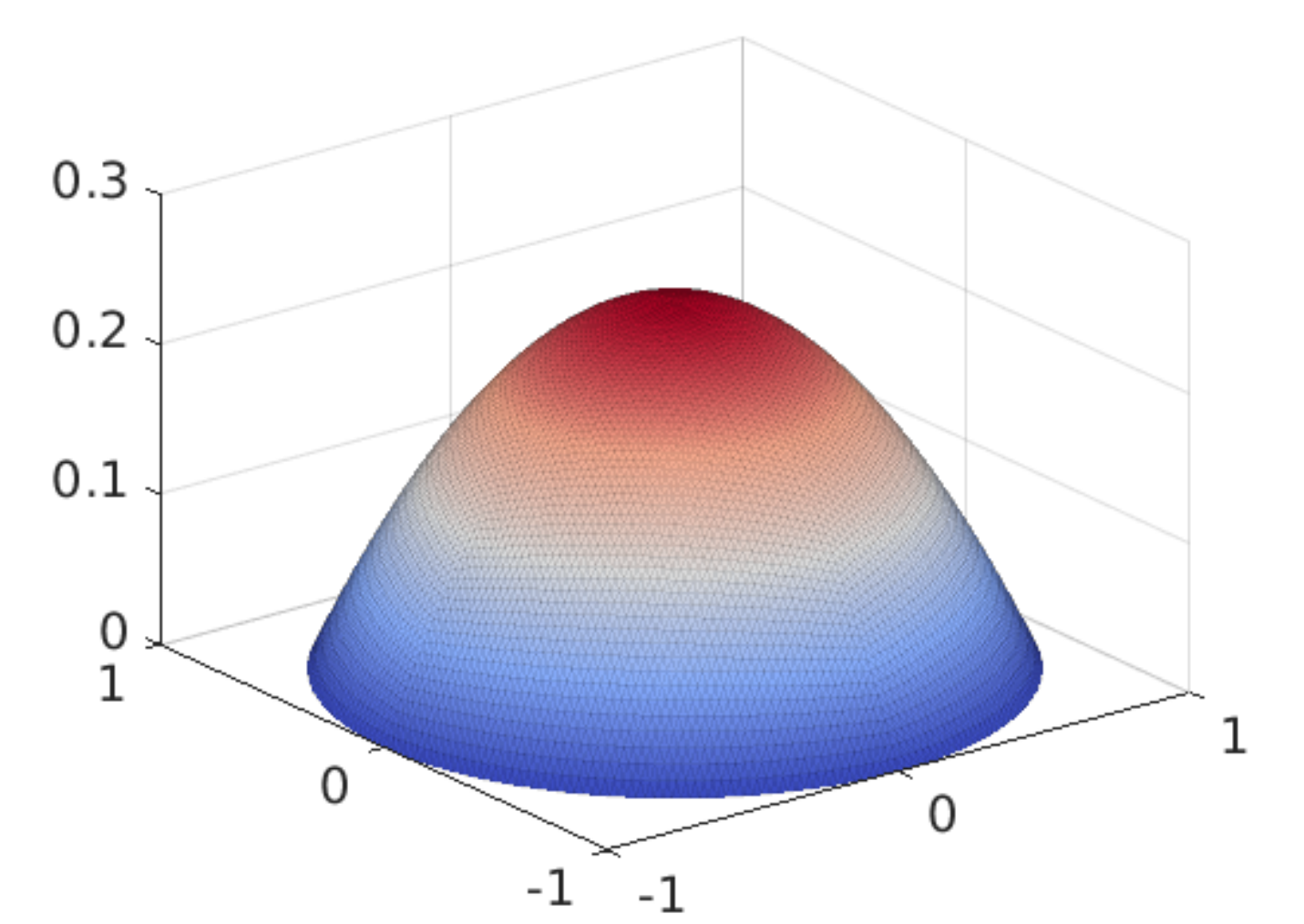}\hfill
\includegraphics[scale=.5, trim= 100 250 100 250, clip]{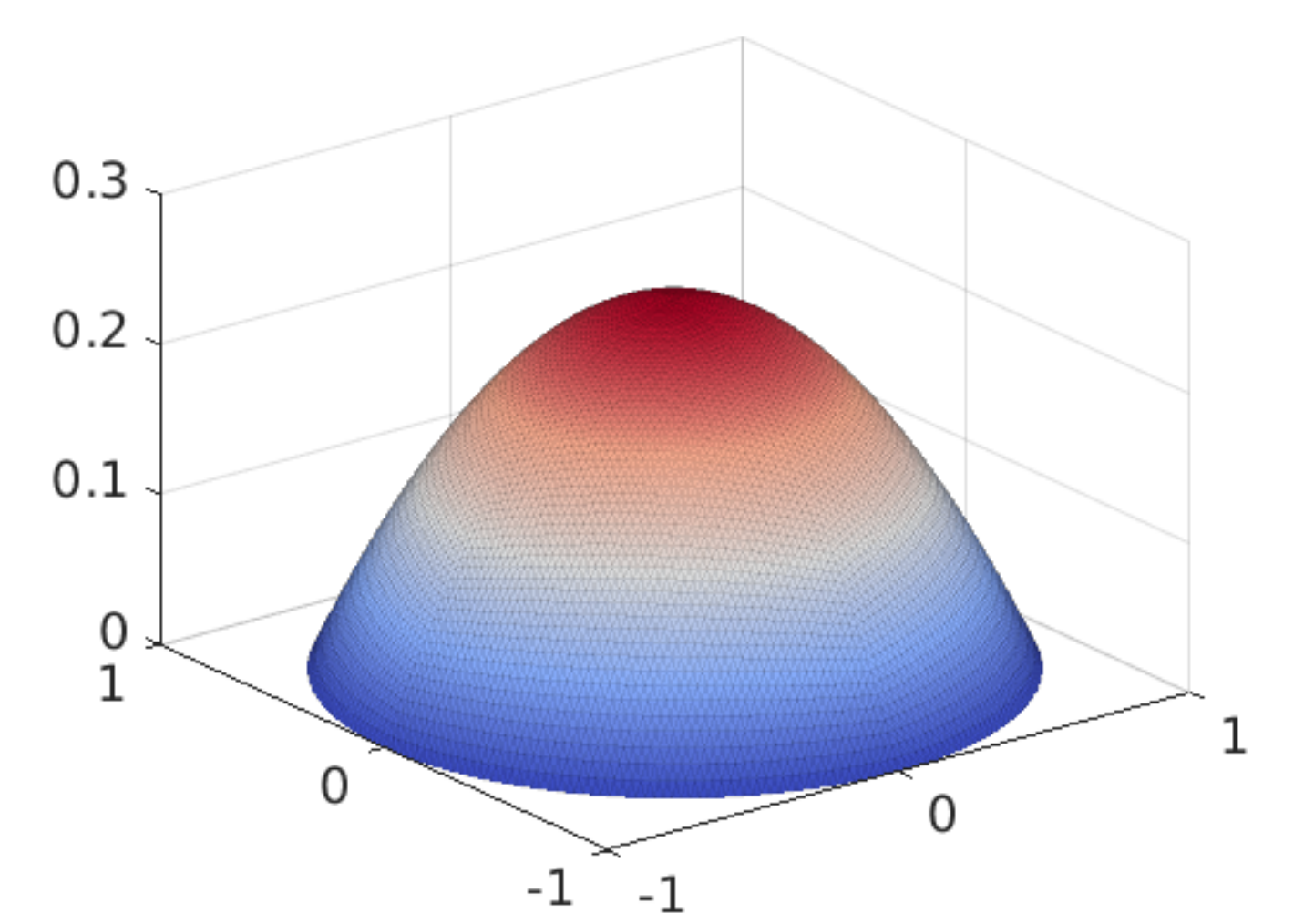}};
\draw (-3.55,3) node{\(\varepsilon=0.01\)};
\draw (3.7,3) node{\(\varepsilon=0.25\)};
\draw (0,0) node{
\includegraphics[scale=.5, trim= 100 250 100 250, clip]{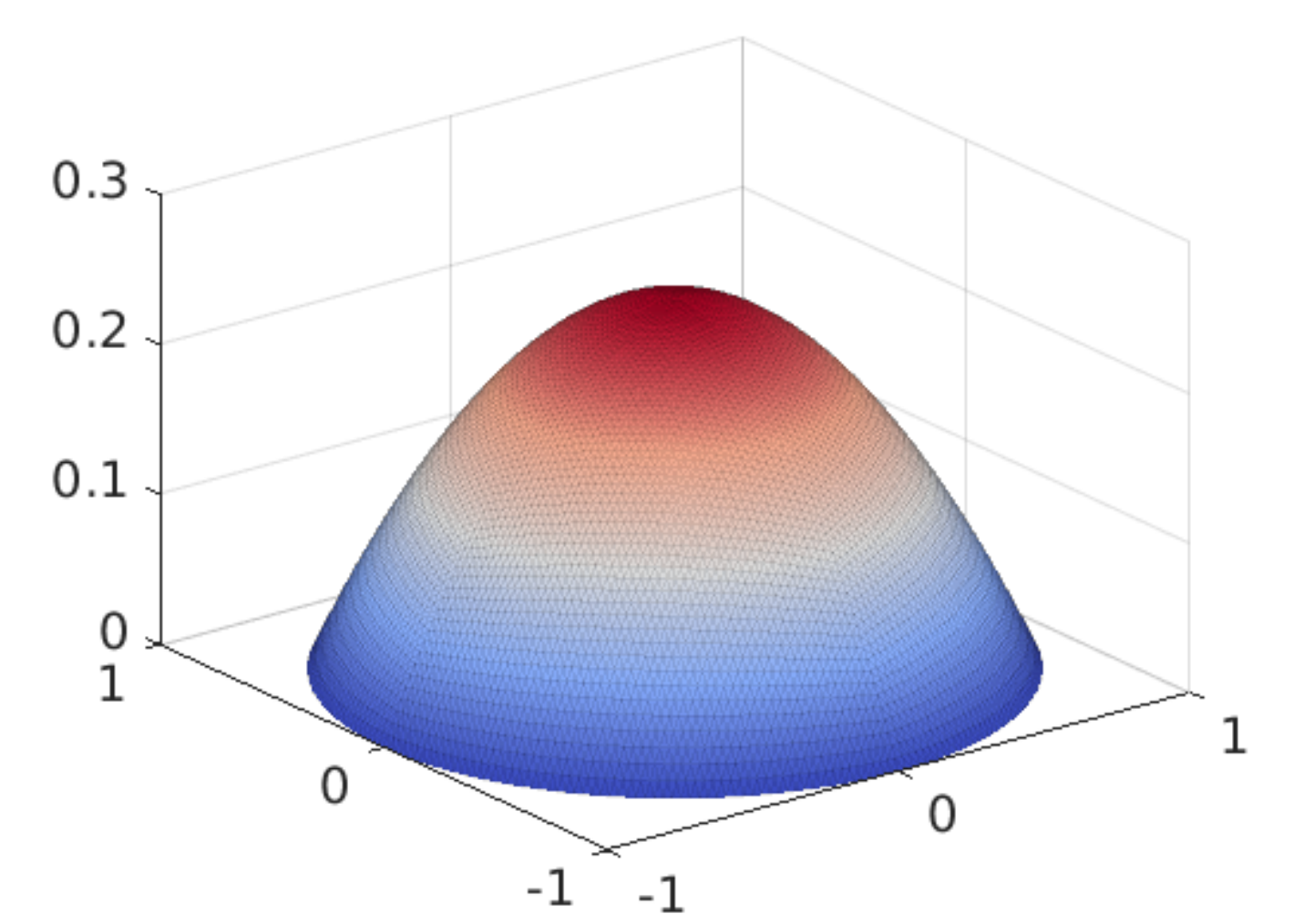}\hfill
\includegraphics[scale=.5, trim= 100 250 100 250, clip]{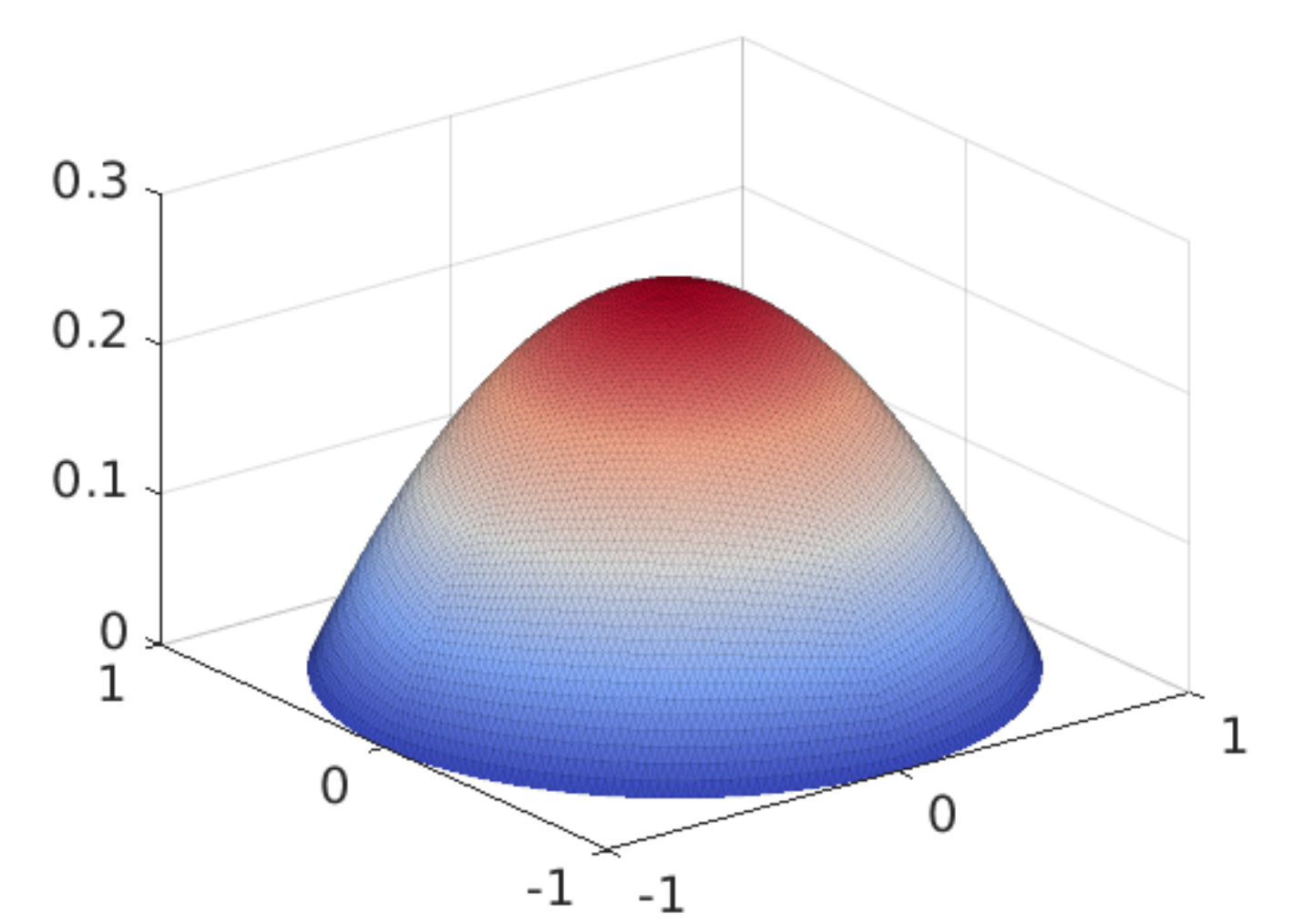}};
\draw (-3.55,-3) node{\(\varepsilon=0.5\)};
\draw (3.7,-3) node{\(\varepsilon=1\)};
\end{tikzpicture}
\caption{\label{fig:meansol}Expectation of the solution \(\hat{u}\) for \(\varepsilon=0.01,0.25,0.5,1\).}
\end{center}
\end{figure}

\begin{figure}[htb]
\begin{center}
\begin{tikzpicture}
\draw (0,6) node{
\includegraphics[scale=.5, trim= 100 250 100 250, clip]{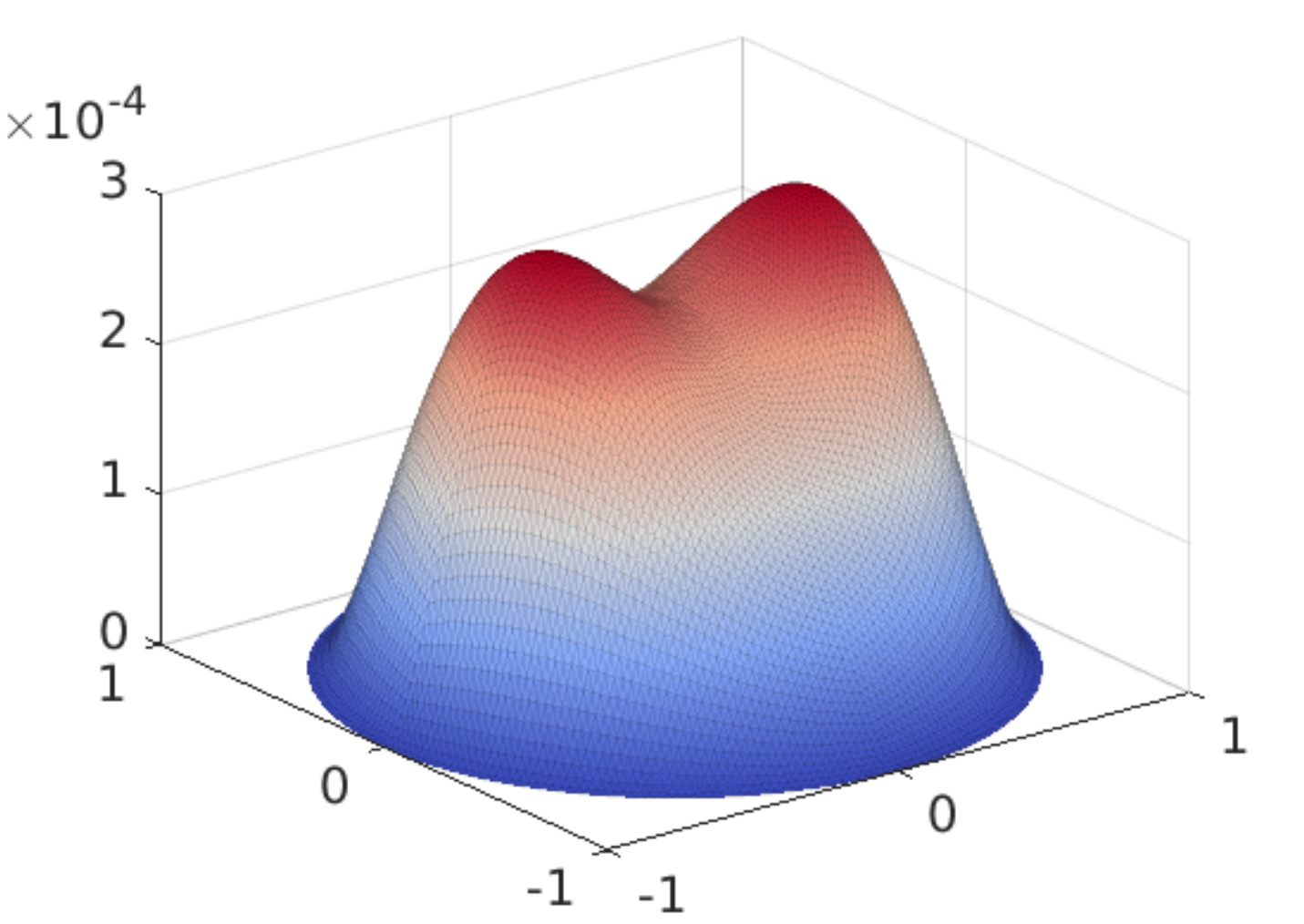}\hfill
\includegraphics[scale=.5, trim= 100 250 100 250, clip]{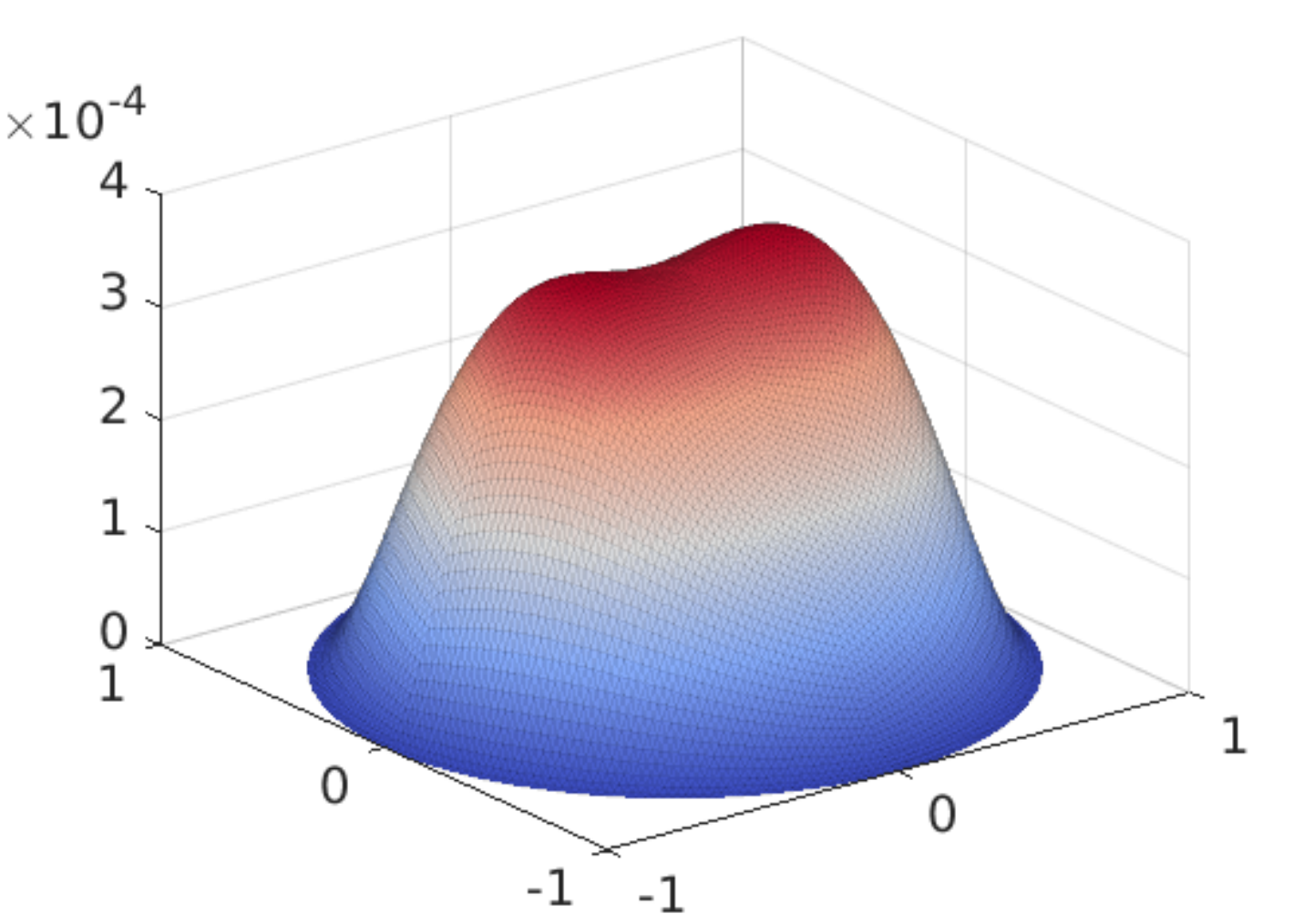}};
\draw (-3.55,3) node{\(\varepsilon=0.01\)};
\draw (3.7,3) node{\(\varepsilon=0.25\)};
\draw (0,0) node{
\includegraphics[scale=.5, trim= 100 250 100 250, clip]{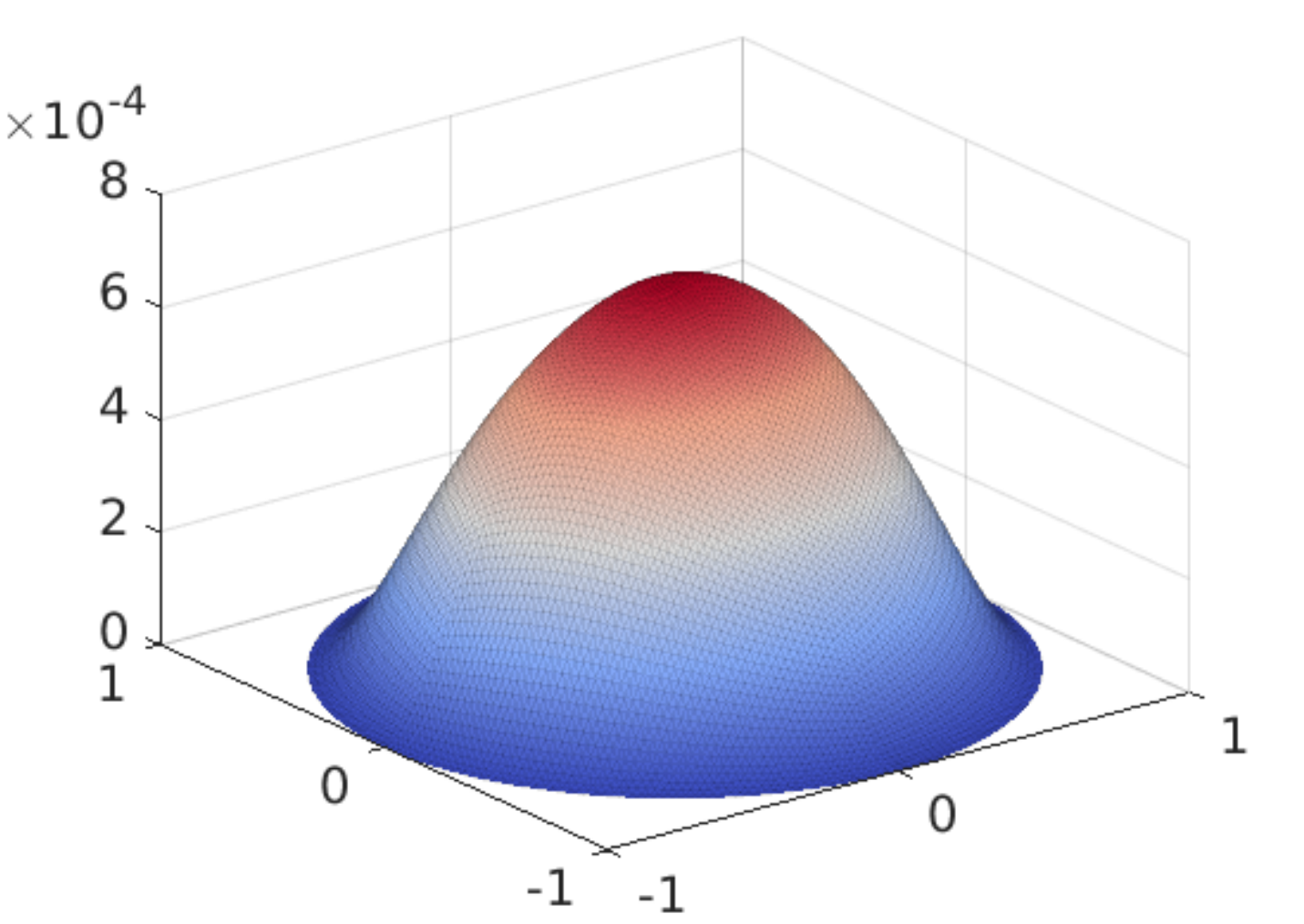}\hfill
\includegraphics[scale=.5, trim= 100 250 100 250, clip]{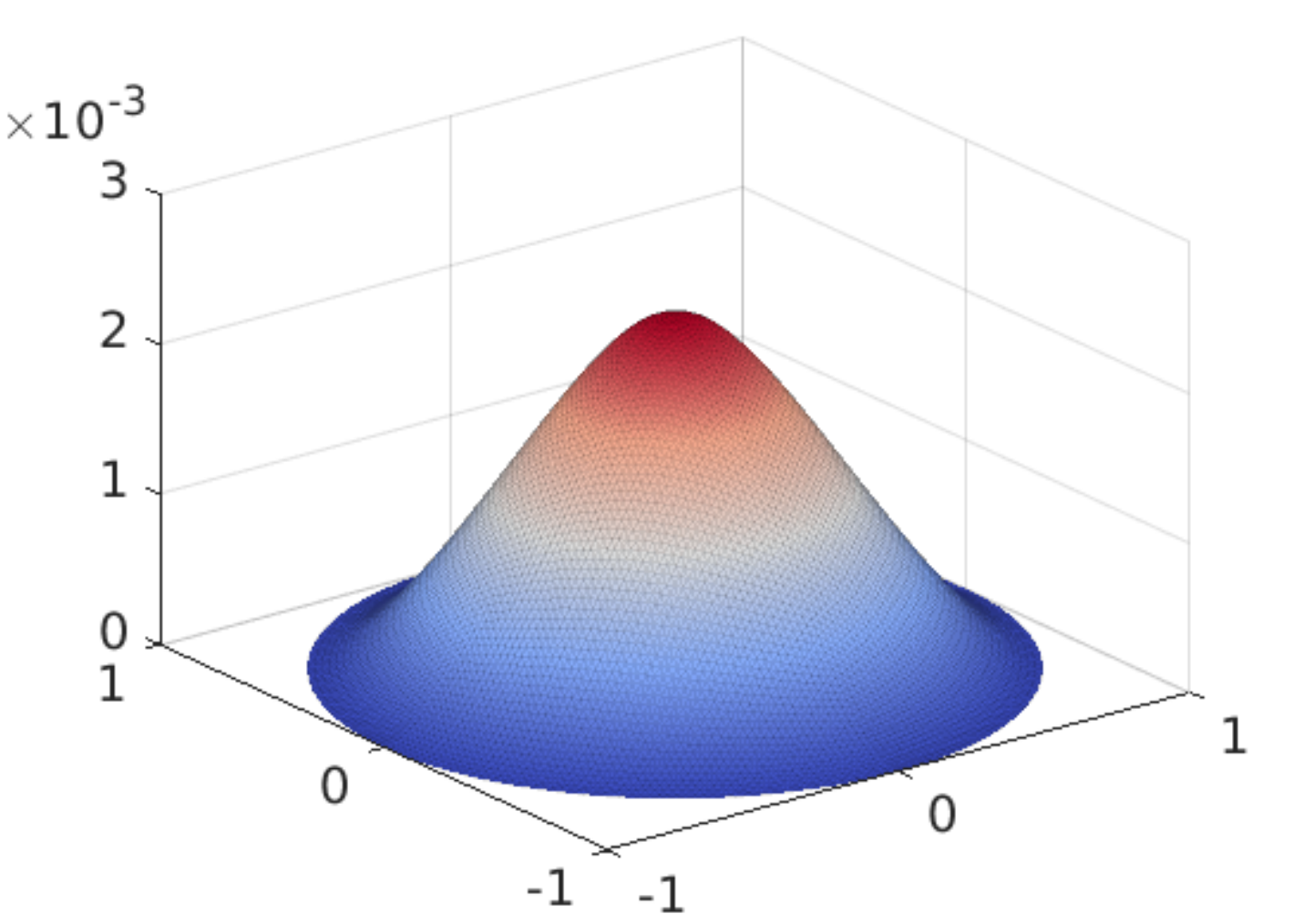}};
\draw (-3.55,-3) node{\(\varepsilon=0.5\)};
\draw (3.7,-3) node{\(\varepsilon=1\)};
\end{tikzpicture}
\caption{\label{fig:varsol}Variance of the solution \(\hat{u}\) for \(\varepsilon=0.01,0.25,0.5,1\).}
\end{center}
\end{figure}

Figure~\ref{fig:meansol} shows visualizations of the expectation for \(\varepsilon=0.01,0.25,0.5,1\).
It turns out that the expectation looks rather similar in all four cases.
This is in contrast to the variance, which is depicted for the same values of \(\varepsilon\) in Figure~\ref{fig:varsol}.
Here, for \(\varepsilon=0.01,0.25\) the shape of the variance is governed by the anisotropy that is induced
by the random vector field, see \eqref{eq:covVField}, while for \(\varepsilon=0.5,1\) the shape is governed by
the shape of the first eigenfunction of the random diffusion coefficient, see Figure~\ref{fig:meanandef}.

For, the computation of \(\E[\hat{u}_0]\) and \(\V[\hat{u}_0]\), cf.\ Theorem~\ref{thm:A}, we are in the setting
that is considered in Section~\ref{sec:analyticDiffusion}, i.e.\ all data are analytic functions.
Consequently, \(\hat{u}_0\) is also an analytic function with respect to the parameters \({\bf z}\in\Gamma^M\).
Hence, we may use the sparse, anisotropic quadrature method based on the Gauss-Legendre points, see \cite{HHPS18}, 
to compute \(\E[\hat{u}_0]\) and \(\V[\hat{u}_0]\) in an efficient manner.\footnote{The 
implementation of the sparse grid quadrature is available on \texttt{https://github.com/muchip/SPQR}.}

In Theorem~\ref{thm:A}, we have derived a point-wise error estimate, thus, we will measure here,
in order to validate this theorem numerically,
\[
\|\E[\hat{u}_\varepsilon]-\E[\hat{u}_0]\|_{H^1(D_\refd)}\quad\text{and}\quad
\|\V[\hat{u}_\varepsilon]-\V[\hat{u}_0]\|_{W^{1,1}(D_\refd)},
\] 
respectively. 
We compute these errors for the values \(\varepsilon=0.03125,0.0625,0.125,0.25,0.5,1\).
For \(\varepsilon=1\), the maximum possible perturbation is approximately \(0.82\), i.e.\
\(\|a_r({\bf x},\omega)\|_{L^\infty(\mathcal{D})}\approx 0.82\) uniformly in \(\omega\in\Omega\).

\begin{figure}[htb]
\begin{center}
\pgfplotsset{width=0.49\textwidth, height=0.49\textwidth}
\begin{tikzpicture}
\begin{loglogaxis}[grid, ymin= 1e-6, ymax = 4e-2, xmin = 3e-2, xmax =1.1, 
  ytick={0.1,0.01,0.001,0.0001,0.00001,0.000001,1e-7},
    legend style={legend pos=south east,font=\small}, legend cell align={left},%
    ylabel={\(L^\infty\)-error}, xlabel ={$\varepsilon$}]
\addplot[line width=0.7pt,color=red,mark=o] table[x index=0,y index=1]{./pictures/errors.txt};\addlegendentry{$\|\E[\hat{u}_0]-\E[\hat{u}_\varepsilon]\|$};
\addplot[line width=0.7pt,color=blue,mark=triangle] table[x index=0,y index=2]{./pictures/errorsg001.txt};\addlegendentry{$\|\V[\hat{u}_0]-\V[\hat{u}_\varepsilon]\|$};
\addplot [line width=0.7pt, color=black,dashed] table[x index={0}, y expr={0.03*x^2}]{./pictures/errorsg001.txt};\addlegendentry{$\varepsilon^2$};
\addplot [line width=0.7pt, color=black,dashed] table[x index={0}, y expr={0.003*x^2}]{./pictures/errorsg001.txt};
\end{loglogaxis}
\end{tikzpicture}
\caption{\label{fig:epsConvergence}Convergence of the perturbation approach with respect to \(\varepsilon\).}
\end{center}
\end{figure}
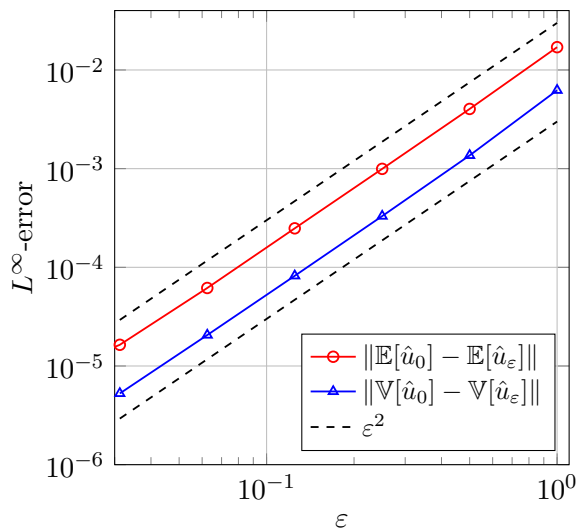

To compute the expectation and the variance of \(\hat{u}_0\) such that the theoretical rate 
of \(\varepsilon^2\) is achieved for all considered values of \(\varepsilon\),
the sparse grid quadrature on level \(q=11\), cf.~\cite{HHPS18}, resulting in \(6859\) quadrature 
points is sufficient.
Figure~\ref{fig:epsConvergence} shows that the theoretical approximation rate of \(\varepsilon^2\) 
in terms of the diffusion coefficient's perturbation's magnitude is perfectly attained in this example.

\section{Conclusion}
For the domain mapping method, 
the presented analysis indicates that smooth data are required in order to derive regularity results for the
solution. Thus quadrature methods relying on the smoothness of the integrand 
may not be feasible to compute quantities of interest if the underlying data are non-smooth.
In the case of small rough perturbations of the diffusion coefficient, a viable alternative  is 
the combination of the domain mapping method 
for the randomly deforming domain with the perturbation approach for the diffusion coefficient. 
To that end, we have derived approximation results based on a first order Taylor expansion of
the diffusion problem's solution with respect to the diffusion coefficient. 
The approximation results guarantee a quadratic approximation of the solution's 
expectation and variance in terms of the perturbations amplitude.
The presented numerical example corroborates this result.
\bibliographystyle{plain}
\bibliography{bibl}
\end{document}